\documentclass{aims}
\usepackage{amsmath}
  \usepackage{paralist}
  \usepackage{graphics} 
  \usepackage{epsfig} 
\usepackage{graphicx}  \usepackage{epstopdf}
 \usepackage[colorlinks=true]{hyperref}
\hypersetup{urlcolor=blue, citecolor=red}

\usepackage{algorithm}
\usepackage{subfigure}
\usepackage{chemarrow}
\usepackage{xcolor}
\usepackage{multirow}

\newtheorem{theorem}{Theorem}[section]
\newtheorem{corollary}{Corollary}

\newtheorem{lemma}[theorem]{Lemma}
\newtheorem{proposition}{Proposition}

\theoremstyle{definition}
\newtheorem{definition}[theorem]{Definition}
\newtheorem{remark}{Remark}

\newtheorem{example}{Example}

\newcommand{\xdag}{x^\dagger}

\title[Two new iterative regularization methods] 
      {Two new non-negativity preserving \\ iterative regularization methods for ill-posed inverse problems}

\author[Ye Zhang and Bernd Hofmann]{}


 \email{ye.zhang@smbu.edu.cn}
 \email{hofmannb@mathematik.tu-chemnitz.de}

\thanks{$^*$ Corresponding author: hofmannb@mathematik.tu-chemnitz.de}

\begin{document}
\maketitle

\centerline{\scshape Ye Zhang}
\medskip
{\footnotesize
 \centerline{Shenzhen MSU-BIT University, 518172 Shenzhen, and}
   \centerline{School of Mathematics and Statistics, Beijing Institute of Technology, 100081 Beijing, China}
} 

\medskip

\centerline{\scshape Bernd Hofmann$^*$}
\medskip
{\footnotesize
 \centerline{Faculty of Mathematics, Chemnitz University of Technology,}
   \centerline{Reichenhainer Str. 39/41, 09107 Chemnitz, Germany}
}

\bigskip


\begin{abstract}
Many inverse problems are concerned with the estimation of non-negative parameter functions. In this paper, in order to obtain non-negative stable approximate solutions to ill-posed linear operator equations in a Hilbert space setting, we develop two novel non-negativity preserving iterative regularization methods. They are based on fixed point iterations in combination with preconditioning ideas. In contrast to the projected Landweber iteration, for which only weak convergence can be shown for the regularized solution when the noise level tends to zero, the introduced regularization methods exhibit strong convergence. There are presented convergence results, even for a combination of noisy right-hand side and imperfect forward operators, and for one of the approaches there are also convergence rates results. Specifically adapted discrepancy principles are used as \emph{a posteriori} stopping rules of the established iterative regularization algorithms. For an application of the suggested new approaches, we consider a biosensor problem, which is modelled as a two dimensional linear Fredholm integral equation of the first kind. Several numerical examples, as well as a comparison with the projected Landweber method, are presented to show the accuracy and the acceleration effect of the novel methods. Case studies of a real data problem indicate that the developed methods can produce meaningful featured regularized solutions.
\end{abstract}

\section{Introduction}

In this paper, we are interested in the stable approximate solution of an ill-posed linear operator equation,
\begin{equation}\label{OperatorEq}
A x = y, \quad x \in L^2_+(\Omega),
\end{equation}
for some bounded domain $\Omega \subset \mathbb{R}^d \;(d=1,2,\cdots)$, which generates the model of a linear inverse problem {\it under non-negativity constraints}. Here, $A$ denotes a compact linear operator acting between the real Hilbert space $L^2(\Omega)$ and an infinite dimensional Hilbert space $\mathcal{Y}$, such that the range $\mathcal{R}(A)$ of the forward operator $A$ is assumed to be an infinite dimensional subspace of $\mathcal{Y}$. In this context, we set
$$L^2_+(\Omega):= \left\{ x\in L^2(\Omega):~ x(t)\geq 0 \textrm{~for almost all~} t\in \Omega \right\}$$
and
$$\bar X:=\{x \in L^2_+(\Omega):\, Ax=y\}. $$
A restriction to solutions from $L^2_+(\Omega)$ is typical for inverse problems in real world applications, where the interested physical quantities are \emph{a priori} known to be non-negative, such as temperatures, material properties or densities. We refer, e.g., to \cite{Lagendijk1988,Piana1997} for a few of such applications in image restoration. Throughout this paper, we denote by $\langle \cdot,\cdot \rangle$ the
inner product in the Hilbert spaces $L^2(\Omega)$ and for simplicity by $\|\cdot\|$ the norm in  $L^2(\Omega)$ as well as the operator norm in the linear space $\mathcal{L}(L^2(\Omega), L^2(\Omega))$ of bounded linear operators.

It is well known that $x \in L^2(\Omega)$ can be decomposed as $x(t)=x_{+}(t)-x_{-}(t) \;(t\in \Omega)$ into a positive part $x_+(t)=\frac{|x(t)|+x(t)}{2}=\max(x(t),0)$
as well as a negative part \linebreak $x_{-}(t)=\frac{|x(t)|-x(t)}{2}=-\min(x(t),0)$, which are both positive functions for almost all $t \in \Omega$ that belong to $L^2_+(\Omega)$. Let us suppose that the (exact) right-hand side $y$ belongs to the restricted range $ A(L^2_+(\Omega))$. This means that there is an element $\bar x \in L^2(\Omega)$ with $A \bar x=y$ and ${\bar x}_{-}=0$, or in other words that the set $\bar X$ of solutions to $Ax=y$ with vanishing negative part is non-empty. Then for injective $A$, the set $\bar X=\{\xdag\}$ is a singleton.
For non-injective $A$ and prescribed arbitrary reference element $x_0 \in L^2(\Omega)$, we denote by $\xdag\in \bar X$ the $x_0$-minimum norm solution with $\|\xdag-x_0\| \le \|\bar x-x_0\|$ for all $\bar x \in \bar X$. The element $\xdag$ is always uniquely determined, because $\bar X$ is a closed and convex subset of $L^2(\Omega)$. Namely, $\bar X$ is the intersection of the linear manifold (shifted subspace) of solutions to $Ax=y,\; x \in L^2 (\Omega),$ and the cone of non-negative elements $x \in L^2 (\Omega)$ with $x(t) \ge 0$ for almost all $t \in \Omega$.

We try to recover in a stable approximate manner the solution $\xdag$ from noisy right-hand sides $y^\delta \in Y$, partially under the additional difficulty that only a perturbed forward operator $A_h$ (for any fixed $h>0$, $A_h \in \mathcal{L}(L^2(\Omega), L^2(\Omega))$ is also assumed to be a compact operator)
is given, where
\begin{equation}\label{ModelError}
\|y^\delta - y\|_{\mathcal{Y}} \leq \delta \quad \mbox{and} \quad \|A_h - A\|_{L^2(\Omega)\to\mathcal{Y}} \leq h
\end{equation}
characterize the deterministic noise model with non-negative noise levels $\delta$ and $h$.

Now we should make some brief remark on the character of ill-posedness of the operator equation (\ref{OperatorEq}). Even if $A$ is a compact linear operator, the classical ill-posedness criterion for linear operator equations having a non-closed range does not apply here, due to the existing non-negativity constraints. However, the concept of {\sl local ill-posedness} (cf.~\cite[Def.~3]{HofPla2018}) is applicable, which was originally developed for nonlinear problems.
To be precise, we learned from \cite[Remark~A4]{EKN1989} that for every element $x_{ref} \in L^2_+(\Omega)$ there are sequences $\{x_n\} \subset L^2_+(\Omega)$ that converge to $x_{ref}$ in the Hilbert space  $L^2(\Omega)$
weakly but not in norm. On the other hand, the compactness of $A$ implies that the sequence $\{Ax_n\}$ converges to $Ax_{ref}$ even in norm. Hence, the equation (\ref{OperatorEq}) is locally ill-posed at all points of $L^2_+(\Omega)$. This requires regularization, because otherwise small amounts of noise in the data and forward model may also lead to arbitrarily large errors in the approximate solutions.
Of course, it should be ensured that the regularized solutions all belong to $L^2_+(\Omega)$, too.

Roughly speaking, two groups of regularization methods exist for solving ill-posed problems (\ref{OperatorEq}): variational regularization methods and iterative regularization methods. Variational regularization methods, i.e., Tikhonov
regularization approaches with general misfit terms, general convex penalties, and convex constraints, including the non-negativity cone, have been intensively studied during the past three decades, see e.g. \cite{FlemmingHofmann2011,Neubauer1988,Tikhonov-1995}. In this work, our focus is on iterative approaches, which are more attractive because of the reduced amount of required computational effort, especially for large-scale inverse problems occurring after an appropriate discretization of the original operator equation (\ref{OperatorEq}).

The most prominent iterative regularization approach for solving (\ref{OperatorEq}) seems to be the projected Landweber method, which is given by the iteration procedure:
\begin{equation}\label{Landweber}
x^{h,\delta}_{k+1} = P_+\left[ x^{h,\delta}_{k} + \omega A^*_h ( y^\delta- A_h x^{h,\delta}_{k} )\right], \quad k=0,1,\cdots,
\end{equation}
with some starting element $x_0 \in L^2_+(\Omega)$ for the iteration and relaxation parameter  $\omega\in(0, 2/\|A_h\|_{L^2(\Omega)\to\mathcal{Y}}^2)$, where $A^*_h$ denotes the adjoint operator of $A_h$, and $P_+$ represents the metric (i.e. pointwise almost everywhere) projection onto $L^2_+(\Omega)$. In the case of exact data and forward operator, i.e. $\delta=h=0$, in contrast to the unconstrained Landweber iteration, we have only weak convergence of the scheme (\ref{Landweber}), cf.~\cite[Theorem 3.2]{Eicke1992}. Strong convergence can only be shown under additional restrictive conditions such as the compactness of $I-\omega A^*A$, cf.~\cite[Theorem 3.3]{Eicke1992}. In the case where $\delta\neq0$ (but $h=0$), we have the estimate $\|x^\delta_{k}-x_{k}\| \leq \omega \|A\|_{L^2(\Omega)\to\mathcal{Y}} k \delta$, cf.~\cite[Theorem 3.4]{Eicke1992}, which can be used to derive the regularization property of the projected Landweber method (\ref{Landweber}) in the weak topology. However, to our best knowledge, even this has not been investigated so far \cite{ClasonKaltenbacherResmerita2019}.

In order to obtain strong convergence of regularized solutions to the exact solution $\xdag$, by using the backward-forward splitting technique, the paper \cite{Eicke1992} proposed the dual projected Landweber iteration for problem (\ref{OperatorEq}):
\begin{equation}\label{LandweberDual}
x_{k} = P_+ A^* w_k, \quad w_{k+1} = w_k + \omega(y - A x_k), \quad \omega\in(0, 2/\|A\|^2)
\end{equation}
for the case of an exact operator and for the noise-free right-hand side of (\ref{OperatorEq}). It has been shown that the strong convergence $x_{k} \to x^\dagger$ as $k\to\infty$ of scheme (\ref{LandweberDual}) holds under the smoothness assumption:
\begin{equation}\label{smoothness}
x^\dagger \in \mathcal{R}(P_+ A^*),
\end{equation}
imposed on the exact solution $\xdag$. However, this source condition is more typical for convergence rates results, see for example the converse results for linear ill-posed problems \cite{Albani2016,Neubauer1997}. This motivates us to develop iterative regularization methods for the ill-posed problems (\ref{OperatorEq}) with non-negative solutions that exhibit strong convergence without using additional smoothness assumptions such as (\ref{smoothness}). To the best of our knowledge, assertions on convergence rates of iterative regularization methods for (\ref{OperatorEq}) with convex constraints, as well as the construction of \emph{a posteriori} stopping rules for such methods are quite limited in the literature. This is mainly due to the fact that regularization procedures under convex constraints become nonlinear, even if the forward operator is linear. Consequently, the standard analysis tools for linear methods cannot be employed. Furthermore, it seems that the combination of non-negativity constraints and perturbed operators has not yet been studied for ill-posed problems (\ref{OperatorEq}). Therefore, in this paper we aim to address all of these issues by introducing a class of new iterative regularization methods based on fixed point iterations in combination with preconditioned ideas.

Instead of imposing the non-negativity constraint during the iteration by projection, one can also construct an iterative method such that the non-negativity of the starting value is preserved during iterations. The most prominent class of such methods is formed by the expectation-maximization (EM) based algorithms, which have been firstly introduced in \cite{Dempster1977} for the approximation for maximum likelihood estimators of problems with incomplete or noisy data. For infinite dimensional ill-posed problems (\ref{OperatorEq}), the regularization property of EM-Kaczmarz type iterative methods has been shown in \cite{Haltmeier2009}. Our methods suggested in this paper also belong to the class of non-negativity preserving iterative methods.

The remainder of the paper is structured as follows: Section~2 is devoted to a fixed point-type equation and its properties, which is a basic tool for the developed approaches. In Section 3, we propose the first novel iterative method for solving the linear operator equation (\ref{OperatorEq}) under non-negativity constraints with both a noisy right-hand side and an inexact forward operator. In this section, we also prove for this method regularization properties and convergence rates results. Section~4 introduces the second novel iterative method, where regularization properties are discussed. In Section 5, the developed iterative methods, equipped with an a posteriori stopping rule, are applied to a biosensor tomography problem. Some numerical examples, as well as a comparison with the projected Landweber method, are presented in Section 6. Finally, concluding remarks are given in Section 7.

\section{A fixed point type equation and its properties}

In order to impose preconditioning effects on the operator equation (\ref{OperatorEq}) or the associated normal equation:
\begin{equation}\label{NormalEq}
A^*A x = A^*y, \quad x \in L^2_+(\Omega)
\end{equation}
Under non-negativity constraints, we introduce and exploit an auxiliary operator, \linebreak $G \in \mathcal{L}(L^2(\Omega),L^2(\Omega))$, which is defined as follows:

\smallskip
{\parindent0em Let} $G: L^2(\Omega) \to L^2(\Omega)$ be a strictly positive definite and self-adjoint bounded linear operator of the multiplication type with
multiplier function $m\in L^\infty(\Omega)$ and constants $\underline m, \overline m>0$, i.e.~we have:
\begin{equation} \label{eq:multiop}
[Gx](t):= m(t)\,x(t), \quad 0<\underline m \le m(t) \le \overline m < \infty, \quad t \in \Omega \;\;\mbox{a.e.}\,.
\end{equation}

{\parindent0em A} natural choice of $G$ and important special case of (\ref{eq:multiop}) is the diagonal operator,
\begin{equation} \label{simpleG}
G=\mu I,
\end{equation}
where $I: L^2(\Omega)\to L^2(\Omega)$ is the identity operator and $\mu>0$ is a given number. For simplicity, the convergence analysis of our two approaches proposed in Sections 3 and 4, respectively, is based on the special case (\ref{eq:multiop}) for the choice $G$. However, in the numerical experiments, we investigate the influence of different choices of $G$ for our algorithms. Similar to arguments in \cite{Piana1997}, the role of $G$ can be interpreted as a preconditioner for accelerating the original Landweber iteration. Numerical simulations in Section 6 indicate that with an appropriate choice of $G$, our iterative regularization algorithms exhibit the acceleration phenomenon.

The equation
\begin{equation}\label{identity_x}
(G + A^*A)x = (G - A^*A)|x| + 2 A^* y, \quad x \in L^2(\Omega)
\end{equation}
with solution set
$$\hat X:= \left\{x \in L^2(\Omega):\, x \;\mbox{obeys the operator equation (\ref{identity_x})} \right\} $$
is the basic tool for the suggested iterative methods in the following sections. It can be rewritten on the one hand as a fixed point equation,
\begin{equation}\label{fixed_point}
x = (G+A^*A)^{-1}[(G - A^*A)|x| + 2 A^* y], \quad x \in L^2(\Omega),
\end{equation}
and on the other hand as:
\begin{equation}\label{both_parts}
A^*A x_+ -G x_- =A^*y, \quad x=x_+-x_-\in L^2(\Omega),
\end{equation}
by using positive and negative parts.
The form (\ref{fixed_point}) is derived from (\ref{identity_x}) by left-side multiplication of the operator $(G+A^*A)^{-1}$, taking into account that by definition, $G$, and consequently $G+A^*A$, are positive definite operators which are both continuously invertible. From the form (\ref{both_parts}), however, we immediately see that the solution set $\hat X$ of equation (\ref{identity_x}) has $\bar X$ as a subset. Namely, every element $\bar x \in \bar X$ has the property $\bar x_-=0$ and therefore fulfils equation (\ref{both_parts}). As Proposition~\ref{pro:equiv} will show, under the stated assumption on $G$, the solution sets $\bar X$ and $\hat X$ do even coincide.

\begin{lemma} \label{lem:normal}
If the element $x^* \in L^2_+(\Omega)$ satisfies the conditions:
\begin{equation} \label{lemcond1}
 L^2_+(\Omega) \ni r:=A^*(Ax^*-y) \quad \mbox{and} \quad \langle x^*,r \rangle=0,
\end{equation}
then this element solves, under the stated assumption $\bar X \not=\emptyset$, the operator equation (\ref{OperatorEq}), i.e.~we have $Ax^*=y$.
\end{lemma}
\begin{proof}
We rewrite the second condition $\langle x^*,r \rangle=0$ as $\langle Ax^*,Ax^*-y \rangle=0$ and equivalently $\|Ax^*\|^2=\langle Ax^*,y \rangle$, which implies that for all $x \in L^2_+(\Omega)$,
\begin{equation*}
\begin{array}{ll}
&  \|Ax-y\|^2-\|Ax^*-y\|^2 =\|Ax\|^2-2\langle Ax,y \rangle-\|Ax^*\|^2+2\langle Ax^*,y \rangle \\
&  \quad =\|Ax\|^2+\|Ax^*\|^2-2\langle Ax,y \rangle =\|A(x-x^*)\|^2+2\langle x,A^*(Ax^*-y) \rangle \ge 0.
\end{array}
\end{equation*}
The last inequality holds since $\langle x,A^*(Ax^*-y) \rangle \ge 0$ by taking into account that both functions $x$ and $A^*(Ax^*-y)=r$ belong to $L^2_+(\Omega)$. This proves the lemma, since
$\|Ax^*-y\|= \min_{x \in L^2_+(\Omega)}\|Ax-y\|\le \|A\xdag-y\|=0.$
\end{proof}

By setting $r:=0$ in Lemma~\ref{lem:normal} and taking into account that $\langle x^*,r \rangle=0$ holds for all solutions $x^*\in L^2_+(\Omega)$ of the normal equation (\ref{NormalEq}),
we immediately obtain the following equivalence assertion from the lemma.

\begin{corollary} \label{cor:normal}
The solution set of the normal equation (\ref{NormalEq}) under non-negativity constraints and the solution set $\bar X$ of  (\ref{OperatorEq}) coincide.
\end{corollary}

Moreover, for all operators $G$ of type (\ref{eq:multiop}), we even find that $\hat X = \bar X$, which also means that all solutions to equation (\ref{identity_x}) have a vanishing negative part.

\begin{proposition} \label{pro:equiv}
Under the stated assumption $\bar X \not=\emptyset$ for the solution set $\bar X$ of equation (\ref{OperatorEq}) and for $G$ from (\ref{eq:multiop}), the operator equations (\ref{OperatorEq}) and (\ref{identity_x}) are equivalent to each other, i.e.~the solution sets $\bar X$ and $\hat X$ coincide.
\end{proposition}
\begin{proof}
As already mentioned, every solution of (\ref{OperatorEq}) also solves (\ref{identity_x}), and we only have to prove the reverse implication. Let $z$ be an arbitrary solution to equation (\ref{identity_x}). We
Use Lemma~\ref{lem:normal} by setting $x^*:=z_+ \in L^2_+(\Omega)$ and $r:=G z_-$. As a consequence of the positive and bounded multiplier function $m$ in the definition of $G$ from (\ref{eq:multiop}), we then have $r \in L^2_+(\Omega)$ and,
$$\langle x^*,r \rangle= \int \limits_{t \in \Omega} m(t)\,z_-(t)\,z_+(t)\,dt =0.$$
Finally we have that $r=A^*(Ax^*-y)$, which is equivalent to $A^*Az_+-Gz_-=A^*y$, is satisfied for any solution $z$ to equation (\ref{identity_x}) due to the rewritten form (\ref{both_parts}). Then Lemma~\ref{lem:normal}
yields the assertion of the proposition and completes the proof.
\end{proof}

As a consequence of this proposition, we can search for approximate solutions to equation (\ref{identity_x}), which has no non-negativity restrictions if we want to
approximate solutions to the original operator equation (\ref{OperatorEq}) under non-negativity constraints. This will be used for the construction of iteration procedures in the subsequent sections.

\section{The first non-negativity preserving iterative regularization method}

\subsection{Algorithm 1 and its regularization property}

In this section we utilize for an iteration process the fixed point equation (\ref{fixed_point}), which is equivalent to (\ref{identity_x}) and (\ref{OperatorEq}), for {\sl compact} and {\sl injective} operators $A: L^2(\Omega)\to \mathcal{Y}$ and $A_h: L^2(\Omega)\to \mathcal{Y}$ with $h\in(0,h_0]$ and some constant $h_0>0$.
As a consequence of the injectivity of $A$, the operator equation (\ref{OperatorEq}) has a unique solution $\xdag$, which is assumed to belong to to $L^2_+(\Omega)$. For the non-injective case, we refer to our second approach proposed in Section 4. For both approaches, we only consider the special case (\ref{simpleG}) for the choice
of the operator $G$ from (\ref{eq:multiop}), i.e. $G=\mu I$ for prescribed $\mu>0$.

Our first iteration approach is described in Algorithm 1. The main goal in the section is to establish basic assertions of regularization theory  for Algorithm~1 in the sense of strong convergence and convergence rates results. To this end, let us first consider in Lemma~\ref{lemma_operator} relevant estimates of quantities involving the inexact forward operator, which will be used for the convergence analysis of both approaches in Sections 3 and 4. Its proof is technical, and can be found in the appendix.

\begin{algorithm}[!htb]
\label{algorithm2}
\caption{The first non-negativity preserving iterative regularization method for ill-posed operator equation (\ref{OperatorEq}) with non-negative solutions.}
\begin{itemize}
\item[] \textbf{Input:} Imperfect forward model $A_h$ and noisy data $y^\delta$ with noise levels $h$ and $\delta$.
\item[] \textbf{Parameters}: $x_0\in L^2(\Omega)$, $k\gets 1$.
\item[] \textbf{While:} $k$ does not satisfy the stopping rule:
\begin{enumerate}
\item[i.]
\begin{equation}\label{Iteration2}
\qquad\qquad
z^{h,\delta}_{k+1} = (G + A^*_h A_h)^{-1} \left[ (G - A^*_h A_h) |z^{h,\delta}_{k} | +  2 A^*_h y^\delta \right].
\end{equation}
\item[ii.] $k \gets k+1$;
\end{enumerate}
\item[] \textbf{Output:} the obtained regularized solution is $x^{h,\delta}_k := |z^{h,\delta}_{k}|$.
\end{itemize}
\end{algorithm}

\begin{lemma}\label{lemma_operator}
There exist three constants $C_{i}(A,G)$ ($i=1,2$) and $h_0(A,G)$, depending only on $\{A,G\}$, such that for all $h\in(0,h_0]$,
\begin{equation}\label{lemma_operator_ineq1}
\hspace{-8mm} \|(G + A^*_hA_h)^{-1} (G - A^*_hA_h) - (G + A^*A)^{-1}(G - A^*A)\|
\leq C_{1}(A,G) h,
\end{equation}
and
\begin{equation}\label{lemma_operator_ineq2}
\hspace{-8mm} \|(G + A^*_hA_h)^{-1} A^*_h - (G + A^*A)^{-1} A^*\| \leq C_{2}(A,G) h.
\end{equation}
\end{lemma}

\medskip

The following convergence theorem uses error estimates including perturbation with respect to noise in the right-hand side of (\ref{OperatorEq}) and with respect to the forward operator.

\begin{theorem}\label{reguThm}
For compact and injective forward operators $A,\;A_h$ and for $G=\mu I$ with some constant $\mu>0$, let $\{x^{h,\delta}_{k}\}\subset L^2_+(\Omega)$ be the sequence generated by Algorithm 1. If the stopping index $k^*=k^*(h,\delta)$ is chosen such that:
\begin{equation}\label{k_condition}
k^*\to \infty \;\textrm{~and~}\; k^*(h+\delta)\to 0 \;\textrm{~as~}\; h,\delta \to 0,
\end{equation}
then the approximate solution $x^{h,\delta}_{k^*}$ converges to $x^\dagger$ in norm of $L^2(\Omega)$ as $\delta,h\to0$.
\end{theorem}

\begin{proof}
By Algorithm 1 and the non-negativity of $x^\dagger$, we have:
\begin{equation*}
\begin{array}{ll}
&  \hspace{-5mm} \|x^{h,\delta}_{k+1}-x^\dagger\| = \| |z^{h,\delta}_{k+1}| -x^\dagger\| \leq \|z^{h,\delta}_{k+1}-x^\dagger\|
\\ & = \| (G  + A^*_hA_h)^{-1} \left[ (G - A^*_hA_h) |z^{h,\delta}_{k} | +  2 A^*_h y^\delta \right] \\ & \qquad - (G + A^*A)^{-1} \left[ (G - A^*A)x^\dagger + 2 A^* y \right] \|
\\ &  = \| (G + A^*_hA_h)^{-1} (G - A^*_hA_h) x^{h,\delta}_{k}  - (G  + A^*_hA_h)^{-1} (G - A^*_hA_h) x^\dagger
\\ & \hspace{-5mm} + (G  + A^*_hA_h)^{-1} (G - A^*_hA_h) x^\dagger - (G + A^*A)^{-1}(G - A^*A)x^\dagger + 2(G + A^*_hA_h)^{-1} A^*_h y^\delta
\\ & \hspace{-5mm} - 2 (G + A^*A)^{-1} A^* y^\delta
 + 2 (G + A^*A)^{-1} A^* y^\delta- 2 (G + A^*A)^{-1} A^* y\|
\\ & \leq \left\| [(G + A^*_hA_h)^{-1} (G - A^*_hA_h)] (x^{h,\delta}_{k}-x^\dagger) \right\|
\\ & \qquad + \|(G + A^*_hA_h)^{-1} (G - A^*_hA_h) - (G + A^*A)^{-1}(G - A^*A)\| \|x^\dagger \|
\\ & \qquad + 2 \|(G + A^*_hA_h)^{-1} A^*_h - (G + A^*A)^{-1} A^*\| \|y^\delta\|+ 2\|(G + A^*A)^{-1} A\|\delta.
\end{array}
\end{equation*}

For any fixed positive number $\delta_0$, denote $C_3= C_{1} \|x^\dagger \| + 2 C_{2} (\|y\| +\delta_0)$. Since $\|(G + A^*A)^{-1} A\|\leq \|G\|^{-1/2}$, we deduce from Lemma \ref{lemma_operator} that for all $\delta\in (0,\delta_0]$ and $h\in(0,h_0]$,
\begin{equation}\label{pfReguIneq1}
\begin{array}{ll}
& \hspace{-8.5mm}  \|x^{h,\delta}_{k}-x^\dagger\| \leq \left\| [(G + A^*_hA_h)^{-1} (G - A^*_hA_h)] (x^{h,\delta}_{k-1}-x^\dagger) \right\| + C_3 h + 2 \|G\|^{-1/2} \delta
\\ & \hspace{-8mm} \leq \left\| [(G + A^*_hA_h)^{-1} (G - A^*_hA_h)]^2 (x^{h,\delta}_{k-2}-x^\dagger) \right\| + 2\left( C_3 h + 2 \|G\|^{-1/2} \delta \right)
\\ & \hspace{-8mm} \leq \cdots \leq \left\| [(G + A^*_hA_h)^{-1} (G - A^*_hA_h)]^k (x_{0}-x^\dagger) \right\| + k\left( C_3 h + 2 \|G\|^{-1/2} \delta \right).
\end{array}
\end{equation}

On the other hand, for the non-negative self-adjoint and compact operator $A_h^*A_h$, denote by $\{\lambda^h_j; u_j\}^\infty_{j=1}$ the eigensystem with the complete orthonormal system $\{u_j\}^\infty_{j=1}$ of eigenelements in $L^2(\Omega)$ associated with the positive eigenvalues $\{\lambda^h_j\}^\infty_{j=1}$ such that $A^*_h A_h u_j = \lambda^h_j u_j$ and $\|A^*_h A_h\|_{L^2(\Omega) \to L^2(\Omega)}=\lambda^h_1\geq \lambda^h_2 \geq \cdots \to0$ as $j\to \infty$. Since $A_h$ is injective, we have the decomposition $x_0-\xdag=\sum^\infty_{j=1} \langle x_0-\xdag, u_j\rangle u_j$, which implies for $G=\mu I$ that:
\begin{equation}\label{ProofSpectral1}
\left[(G + A^*_hA_h)^{-1} (G - A^*_hA_h)\right]^k (x_0-\xdag) = \sum^\infty_{j=1} \left(\frac{\mu - \lambda^h_j}{\mu + \lambda^h_j} \right)^k \langle x_0-\xdag, u_j\rangle  u_j
\end{equation}
by noting that the operators $(\mu I+ A^*_hA_h)$ and $(\mu I - A^*_hA_h)$ commute.

Note that for each fixed $k$, $\left| \left(\frac{\mu - \lambda^h_j}{\mu + \lambda^h_j} \right)^{2k} \langle x_0-\xdag, u_j\rangle^2 \right|\leq \langle x_0-\xdag, u_j\rangle^2$
and that we have for $x_0-\xdag\in L^2(\Omega)$,
\begin{equation}\label{ProofSpectral2}
\sum^\infty_{j=1} \langle x_0-\xdag, u_j\rangle^2< \infty.
\end{equation}
Then, according to Lebesgue's dominated convergence theorem, (\ref{ProofSpectral2}) implies that the series $\sum^\infty_{j=1} \left(\frac{\mu - \lambda^h_j}{\mu + \lambda^h_j} \right)^{2k} \langle x_0-\xdag, u_j\rangle^2$ is uniformly convergent with respect to $k$. Consequently, we deduce that:
\begin{equation}\label{ProofSpectral3}
\hspace{-3mm} \lim_{k\to\infty} \sum^\infty_{j=1} \left(\frac{\mu - \lambda^h_j}{\mu + \lambda^h_j} \right)^{2k} \langle x_0-\xdag, u_j\rangle^2 = \sum^\infty_{j=1} \lim_{k\to\infty} \left(\frac{\mu - \lambda^h_j}{\mu + \lambda^h_j} \right)^{2k} \langle x_0-\xdag, u_j\rangle^2 =0
\end{equation}
by noting that $\left| \frac{\mu - \lambda^h_j}{\mu + \lambda^h_j}\right|<1$ holds for all fixed $\lambda^h_j>0$ and $\mu>0$.

Finally, by combining (\ref{pfReguIneq1}), (\ref{ProofSpectral3}) and the choice of $k^*$ in (\ref{k_condition}), we conclude that:
\begin{equation}\label{PfIneq0}
\begin{array}{ll}
&  \|x^{h,\delta}_{k^*}-x^\dagger\| \\ & \quad \leq \left\{ \sum\limits^\infty\limits_{j=1} \left(\frac{\mu - \lambda^h_j}{\mu + \lambda^h_j} \right)^{2k^*} \langle x_0-\xdag, u_j\rangle^2  \right\}^{1/2} + k^* \left( C_3 h + 2 \|G\|^{-1/2} \delta \right) \to 0
\end{array}
\end{equation}
as $h,\delta \to 0$.
\end{proof}

According to the inequality (\ref{PfIneq0}), it is clear that the following assertion holds true.
\begin{corollary} \label{cor:noiseFree}
Let $(x_{k})_k\subset L^2_+(\Omega)$ be the sequence generated by Algorithm 1 with $G=\mu I$ and noise-free information $\{A, y\}$. Then, $x_{k}$ converges to $x^\dagger$ as $k\to\infty$.
\end{corollary}

\begin{proposition}\label{estimateNoisyMethod1}
Let $(z^{h,\delta}_{k})^\infty_{k=1}$ and $(z_{k})^\infty_{k=1}$ be the sequences generated by Algorithm 1 with noisy information $\{A_h, y^\delta\}$ and exact information $\{A, y\}$, respectively. Then, there exists a constant $C_4$, depending only on $\{A,G\}$, such that for $h\in(0,h_0]$,
\begin{equation}\label{lemma_operator_Method1}
\|z^{h,\delta}_{k}  - z_{k} \| \leq C_4 (h+\delta) k.
\end{equation}
\end{proposition}

\begin{proof}
Setting $h=\delta=0$ in (\ref{pfReguIneq1}), we obtain, for all $k\in \mathbb{N}$,
\begin{equation*}
\|x_{k}\|\leq \|\bar{x}\| + \|x_{k}-\bar{x}\| \leq \|\bar{x}\| + \|x_{k-1}-\bar{x}\| \leq \cdots \leq \|\bar{x}\| + \|x_{0}-\bar{x}\|,
\end{equation*}
which implies that:
\begin{equation*}
\begin{array}{ll}
& \|z^{h,\delta}_{k}-z_{k}\| = \Big\|  (G  +  A^*_hA_h)^{-1} \left[ (G - A^*_hA_h) |z^{h,\delta}_{k-1} | +  2 A^*_h y^\delta \right] \\
& \qquad \qquad - (G  +  A^*A)^{-1} \left[ (G - A^*A) |z_{k-1} | +  2 A^* y \right] \Big\| \\
& \quad \leq\|(G  +  A^*_hA_h)^{-1} (G - A^*_hA_h) (|z^{h,\delta}_{k-1} | - |z_{k-1}| )\| \\
& \qquad \qquad + \| [(G  +  A^*_hA_h)^{-1} (G - A^*_hA_h) - (G  +  A^*A)^{-1} (G - A^*A)] |z_{k-1}| \|  \\
& \qquad \qquad + 2 \|(G  +  A^*_hA_h)^{-1} A^*_h y^\delta - (G  +  A^*_hA_h)^{-1} A^*_h y \|
\\
& \qquad \qquad + 2 \|(G  +  A^*_hA_h)^{-1} A^*_h y - (G  +  A^*A)^{-1} A^* y \|  \\
& \quad \leq  \|z^{h,\delta}_{k-1}  - z_{k-1}\| + C_1 (\|\bar{x}\| + \|x_{0}-\bar{x}\|) h + \frac{4}{\sqrt{\|G\|}} \delta + 2 C_2 \|y\| h \\
& \quad \leq \cdots \leq \|z^{h,\delta}_{0}  - z_{0}\| + k \left[ C_1 (\|\bar{x}\| + \|x_{0}-\bar{x}\|) h + \frac{4}{\sqrt{\|G\|}} \delta + 2 C_2 \|y\| h \right] .
\end{array}
\end{equation*}
We complete the proof by noting $z^{h,\delta}_{0}  = z_{0} = x_0$ and defining $C_4=\max\{ C_1 (\|\bar{x}\|$ $+ \|x_{0}-\bar{x}\|) + 2 C_2 \|y\|, 4/\sqrt{\|G\|} \}$.
\end{proof}

For real-world inverse problems, the a posteriori stopping rules for iterative regularization methods are definitely more attractive, as they do not require any a priori knowledge of unknown exact solutions. They are also usually related  to  the residual error for the desired approximate solution, which represents one of the most important benchmark indices for evaluating the developed method in practice. In order to propose the suitable discrepancy principle for our approach, we define:
\begin{equation}\label{DefGh}
\tilde{G}:= \mu \,I_\mathcal{Y},
\end{equation}
where $I_\mathcal{Y}$ is the identity operator in $\mathcal{Y}$. Clearly, we have for $\tilde G$ from (\ref{DefGh}) and $G$ from (\ref{simpleG}) the identity,
\begin{equation}\label{IneqG}
(G + A^*_h A_h)^{-1} A^*_h = A^*_h (\tilde{G} + A_h A^*_h)^{-1},
\end{equation}
taking into account that for all $h>0$ the operator $\tilde{G} + A_h A^*_h$ is self-adjoint and positive definite and so is the inverse $(\tilde{G} + A_h A^*_h)^{-1}
\in \mathcal{L}(\mathcal{Y},\mathcal{Y})$.

Now we are in the position to propose the modified discrepancy principle: find $k^*$ such that for all $k \leq  k^*$,
\begin{equation}\label{Method1discrepancy}
\|(\tilde{G} + A_h A^*_h)^{-1} ( y^\delta - A_h |z^{h,\delta}_{k^*}|)\|_{\mathcal{Y}} \leq \tau(\delta + h C^\dagger)  \leq  \|(\tilde{G} + A_h A^*_h)^{-1} ( y^\delta - A_h |z^{h,\delta}_{k}|)\|_{\mathcal{Y}},
\end{equation}
where $\tau$ and $C^\dagger$ are two fixed numbers such that $\tau>1/\mu$ and $C^\dagger\geq\|x^\dagger\|$.

\begin{theorem}\label{reguThmPosteriori}
Let $x^{h,\delta}_{k^*}\in L^2_+(\Omega)$ be the approximate solution, obtained by Algorithm 1 with the stopping rule (\ref{Method1discrepancy}). Then, $x^{h,\delta}_{k^*}$ converges to $x^\dagger$ as $\delta,h\to0$.
\end{theorem}

\begin{proof}
Let $(h_n, \delta_n)$, $n=1,2,\cdots$, be a sequence converging to $(0,0)$ as $n\to \infty$, and let $A_n:=A_{h_n}$ and $y^n:=y^{\delta_n}$ be the corresponding sequences of perturbed forward operator and data. For each quaternion $(h_n, \delta_n, A_n, y^n)$, denote by $k_n=k^*(h_n, \delta_n, A_n, y^n)$ the corresponding stopping index determined from the discrepancy principle (\ref{Method1discrepancy}). In order to prove the convergence of $x^{h_n,\delta_n}_{k_n}$, we follow the idea, proposed in \cite[Theorem 2.4]{HankeNeubauerScherzer1995} for nonlinear ill-posed problems, and distinguish two cases: (i) $k_n$ is finite for all $h,\delta> 0$, and (ii) $k_n\to \infty$ when $n\to \infty$.

We start with case (i).  Let $\tilde{k}$ be a finite accumulation point of $k_n$. Then, there exists a constant $n_0$ such that for all $n\geq n_0$: $k_n=\tilde{k}$. From the definition of $k_n$ it follows that:
\begin{equation}\label{PfposterioriIneq1}
\|(\tilde{G} + A_n A^*_n)^{-1} (A_n |z^{h_n,\delta_n}_{\tilde{k}}|-y_n)\|_{\mathcal{Y}}  \leq \tilde{k} \tau(h_n+\delta_n).
\end{equation}
Since $z^{h,\delta}_{\tilde{k}}$ depends continuously on $(A_h, y^\delta)$ as $\tilde{k}$ is fixed now, we also have:
\begin{equation}\label{PfposterioriIneq2}
z^{h_n,\delta_n}_{\tilde{k}} \to z_{\tilde{k}} \textrm{~and~} A_n |z^{h_n,\delta_n}_{\tilde{k}}| \to A |z_{\tilde{k}}| \quad \textrm{~as~} n\to \infty.
\end{equation}
Taking the limit in (\ref{PfposterioriIneq1}) yields $A x_{\tilde{k}} = A |z_{\tilde{k}}| =y$. Thus, $x_{\tilde{k}}$ is a solution of (\ref{OperatorEq}), and with (\ref{PfposterioriIneq2}), we obtain $x^{h_n,\delta_n}_{k_n} \to x_{\tilde{k}}$ as $n\to\infty$.

Now, let's consider case (ii). Without loss of generality, we assume that $k_n$ increases monotonically with $n$. Then, by using (\ref{IneqG}) and the inequality $z^{h,\delta}_{k+1}- |z^{h,\delta}_{k}| = 2(G + A^*_h A_h)^{-1}  A^*_h( y^\delta - A_h |z^{h,\delta}_{k}|)$, we deduce together with the positivity of $x^\dagger$ that.
\begin{equation*}
\begin{array}{ll}
& \|z^{h,\delta}_{k+1}- x^\dagger\|^2 - \|z^{h,\delta}_{k}- x^\dagger\|^2 \\
&  \leq \|z^{h,\delta}_{k+1}- x^\dagger\|^2 - \||z^{h,\delta}_{k}|- x^\dagger\|^2 = 2(|z^{h,\delta}_{k}|-x^\dagger, z^{h,\delta}_{k+1}- |z^{h,\delta}_{k}|) + \|z^{h,\delta}_{k+1}-|z^{h,\delta}_{k}|\|^2 \\
&  = 4 \left( |z^{h,\delta}_{k}|-x^\dagger, (G + A^*_h A_h)^{-1}  A^*_h( y^\delta - A_h |z^{h,\delta}_{k}|) \right) \\
& \qquad + 4 \|(G + A^*_h A_h)^{-1}  A^*_h( y^\delta - A_h |z^{h,\delta}_{k}|)\|^2 \\
&  = 4 \left( A_h (|z^{h,\delta}_{k}|-x^\dagger) , (\tilde{G} + A_h A^*_h)^{-1} (y^\delta - A_h |z^{h,\delta}_{k}| ) \right) \\
& \qquad  + 4\|(G + A^*_h A_h)^{-1}  A^*_h( y^\delta - A_h |z^{h,\delta}_{k}|)\|^2  \\
&  =  4\left( (\tilde{G} + A_h A^*_h)^{-1} (y^\delta - A_h |z^{h,\delta}_{k}| ) ,  y^\delta -A_h x^\dagger \right) \\
&  \qquad - 4 \left( (\tilde{G} + A_h A^*_h)^{-1} (y^\delta - A_h |z^{h,\delta}_{k}| ) , \tilde{G} (\tilde{G} + A_h A^*_h)^{-1} (y^\delta - A_h |z^{h,\delta}_{k}| \right)  \\
&  \leq  4 \| (\tilde{G} + A_h A^*_h)^{-1} (y^\delta - A_h |z^{h,\delta}_{k}| )\|  \|y^\delta -A_h x^\dagger \|  \\
&  \qquad - 4\mu \| (\tilde{G} + A_h A^*_h)^{-1} (y^\delta - A_h |z^{h,\delta}_{k}| )\|^2,
\end{array}
\end{equation*}
which implies together with the inequality $\|y^\delta -A_h x^\dagger\|\leq \delta + h C^\dagger$ that,
\begin{equation}\label{PfposterioriIneq3}
\begin{array}{ll}
& \|z^{h,\delta}_{k+1}- x^\dagger\|^2 \leq \|z^{h,\delta}_{k}- x^\dagger\|^2 + 4 \|(\tilde{G} + A_h A^*_h)^{-1} ( y^\delta - A_h |z^{h,\delta}_{k}|)\|_{\mathcal{Y}}  \\
& \qquad\qquad\qquad\quad  \cdot \left\{ \delta + h C^\dagger - \mu \|(\tilde{G} + A_h A^*_h)^{-1} ( y^\delta - A_h |z^{h,\delta}_{k}|)\|_{\mathcal{Y}} \right\},
\end{array}
\end{equation}
By the choice of $k^*$ in (\ref{Method1discrepancy}) and the assumption that $k^*\to \infty$ as $h,\delta\to0$, we deduce that for $n>m$,
\begin{equation}\label{PfIneqPosteriori}
\|z^{h_n,\delta_n}_{k_n}- x^\dagger\| \leq \|z^{h_n,\delta_n}_{k_{n-1}}- x^\dagger\| \leq \cdots \leq \|z^{h_n,\delta_n}_{k_m}- z_{k_m}\| + \|z_{k_m}- x^\dagger\|.
\end{equation}
From Corollary \ref{cor:noiseFree} we deduce that we can fix $m$ to be so large that the last term on the right-hand side of (\ref{PfIneqPosteriori}) is sufficiently close to zero; now that $k_m$ is fixed, we can apply (\ref{PfposterioriIneq2}) to conclude that the left-hand side of (\ref{PfIneqPosteriori}) must go to zero when $n\to\infty$.

Finally, let's show the stopping rule (\ref{Method1discrepancy}) is well-posed, i.e. $k^*<\infty$ for any fixed $(h,\delta)\in \mathbb{R}_+\times\mathbb{R}_+$. Otherwise, there exists a pair of positive numbers $(h,\delta)$ such that for all $k$,
\begin{equation}\label{Method1discrepancy2}
\|(\tilde{G} + A_h A^*_h)^{-1} ( y^\delta - A_h |z^{h,\delta}_{k}|)\|_{\mathcal{Y}} > \tau(\delta + h C^\dagger).
\end{equation}

Combine (\ref{PfposterioriIneq3}) and (\ref{Method1discrepancy2}) to obtain that for $k>\frac{\|x_{0}- x^\dagger\|^2 }{4 \tau (\tau\mu-1)(\delta + h C^\dagger)^2}$,
\begin{equation*}\label{PfposterioriIneq4}
\begin{array}{ll}
& \|z^{h,\delta}_{k}- x^\dagger\|^2 \leq \|z^{h,\delta}_{k-1}- x^\dagger\|^2 - 4 \tau (\tau\mu-1)(\delta + h C^\dagger)^2   \\ & \qquad \leq \cdots \leq \|x_{0}- x^\dagger\|^2 - 4 k \tau (\tau\mu-1)(\delta + h C^\dagger)^2<0,
\end{array}
\end{equation*}
which contradicts the assumption (\ref{Method1discrepancy2}).

\end{proof}

\begin{remark}
According to the proof of Theorem \ref{reguThmPosteriori}, we know that the strong convergence of Algorithm 1 can be obtained according to the stopping rule (\ref{Method1discrepancy}) without assumptions on the injectivity of forward operators. These facts will be emphasized in Section 4 (cf. Theorem \ref{ThmRegu1}) for our second approach.
\end{remark}

\subsection{Convergence rates under the a priori stopping rule}

By the proof of Theorem \ref{reguThm}, cf. (\ref{PfIneq0}), we know that the regularization error of Algorithm 1 can be controlled by an explicit linear formula, which allows us to derive convergence rate results by using the general linear regularization theory. The goal in this subsection is to realize this idea. For investigating the convergence rates of ill-posed problems, additional smoothness assumptions on the exact unknown solution $x^\dagger$ in correspondence with the forward operator and the regularization method under consideration should be satisfied. Otherwise, the rate of convergence of $x^{h,\delta}_{k^*(h,\delta)} \to x^\dagger$ as $\delta,h\to0$ might be arbitrarily slow (cf.~\cite{Schock}). In this work, we consider the following range-type source condition
\begin{equation}\label{SC}
x_0-x^\dagger = \varphi(A^*A) v,
\end{equation}
where $\varphi: \mathbb{R}_+ \to \mathbb{R}_+$ is a index function\footnote{A function $\varphi: \mathbb{R}_+ \to \mathbb{R}_+$ is called an index function if it is continuous, strictly increasing and satisfies $\lim_{\lambda\to 0+} \varphi(\lambda)=0$.}. The concept of index function extends the conventional H\"older-type source condition of exact solutions to a more general source condition, including the logarithmic source conditions for exponential ill-posed problems.  This terminology has been frequently used recently for studying convergence rate results, see e.g. \cite{ZhangHof2018,ZhangHof2019}.

Let's first consider the power-type source conditions, i.e.
\begin{equation}\label{HolderSource1}
\varphi_p(\lambda)=\lambda^p, \quad p>0.
\end{equation}
For the concision of the discussion, we introduce:
\begin{equation}\label{specFuns}
g_k(\lambda):= \left( \frac{\mu - \lambda}{\mu + \lambda} \right)^k,
\end{equation}
and for the compact operator $A^*A$ with singular system $\{\lambda_j; u_j\}^\infty_{j=1}$, we define: (see \cite[\S~34]{Helmberg1969} for a more rigorous definition for general continuous function $g_k(\lambda)$ and arbitrary bounded self-adjoint operators)
\begin{equation}\label{specFuns}
g_k(A^*A)x:= \sum^\infty_{j=1} \left(\frac{\mu - \lambda_j}{\mu + \lambda_j} \right)^k \langle x, u_j\rangle  u_j, \qquad \forall x\in L^2(\Omega).
\end{equation}

The following lemma indicates that function $g_k$ can be viewed as a \emph{qualification} of regularization methods \cite{Mathe-2003,ZhangHof2018,ZhangHof2019}.

\begin{lemma}\label{Lemma_qualification1}
There exists a positive constant $\gamma_1$ such that:
\begin{equation}\label{qualification1}
\sup_{\lambda\in(0,\|A^*A\|]} \varphi_p(\lambda) g_k(\lambda) \leq \gamma_1 \varphi_p(k^{-1}).
\end{equation}
\end{lemma}

\begin{proof}
The lemma follows from the following inequalities:
\begin{equation*}
\begin{array}{ll}
& \lambda^p \left( \frac{\mu - \lambda}{\mu + \lambda} \right)^k \leq_{\lambda= \frac{p\mu}{k+ \sqrt{p^2+k^2}}} \left( \frac{p\mu}{k+ \sqrt{p^2+k^2}} \right)^p \left( \frac{k+ \sqrt{p^2+k^2}-p}{k+ \sqrt{p^2+k^2}+p} \right)^k \leq \left( \frac{p\mu}{2} \right)^p k^{-p}.
\end{array}
\end{equation*}
\end{proof}

Before discussing the convergence rates results, we need the following lemma, which can be obtained by Lemma \ref{lemma_operator} and \cite[\S~4.1, Lemma 1.1]{Vainikko1986}.

\begin{lemma}\label{Lemma_SC1Err}
Let $A$ and $A_h$ satisfy the inequality (\ref{ModelError}). Then, there exists a positive number $C_5=C_5(A,p,\|v\|)$ such that:
\begin{equation}\label{SC1ErrIneq}
\| (A^*_h A_h)^p v - (A^*A)^p v \|\leq C_5 h^{\min(1,p)}.
\end{equation}
\end{lemma}

\begin{theorem}\label{reguThm_priori1}
Let $(x^{h,\delta}_{k})^\infty_{k=1}\subset L^2_+(\Omega)$ be the sequence generated by Algorithm 1. Then, under the source condition (\ref{HolderSource1}), if the iterative number is chosen by $k^*\sim (h^{\min(1,p)}+\delta)^{-1/(p+1)}$~\footnote{$k\sim f(h,\delta)$ means that $k= \tilde{C} \cdot f(h,\delta)$ with some positive number $\tilde{C}$.}, we have the convergence rate:
\begin{equation*}\label{prioriRate1}
\|x^{h,\delta}_{k^*}-x^\dagger\| = \mathcal{O}\left( (h^{\min(1,p)}+\delta)^{\frac{p}{p+1}} \right) \textrm{~as~} h,\delta\to 0.
\end{equation*}
\end{theorem}

\begin{proof}
It follows from Theorem \ref{reguThm}, Lemmas \ref{Lemma_qualification1} and \ref{Lemma_SC1Err} and source condition (\ref{HolderSource1}) that:
\begin{equation}\label{reguThm_priori1ProofIneq}
\begin{array}{ll}
& \hspace{-2mm} \|x^{h,\delta}_{k}-x^\dagger\| \\ &
\hspace{-2mm} \leq \left\| g_k(A^*_hA_h) \left\{ (A^*_hA_h)^p v + [(A^*A)^p -(A^*_hA_h)^p] v \right\} \right\| + k \left( C_3 h + \frac{2}{\sqrt{\|G\|}} \delta \right)
\\ & \hspace{-2mm} \leq \left\| g_k(A^*_hA_h) (A^*_hA_h)^p v \right\| + \left\| [(A^*A)^p -(A^*_hA_h)^p] v \right\|  + k \left( C_3 h + \frac{2}{\sqrt{\|G\|}} \delta \right)
\\ & \hspace{-2mm} \leq \gamma_1 k^{-p} + C_5 h^{\min(1,p)}+ k \left( C_3 h + \frac{2}{\sqrt{\|G\|}} \delta \right),
\end{array}
\end{equation}
which yields the required convergence rate by using the definition of $k^*$.
\end{proof}

Now, let's consider the logarithmic source conditions, i.e.
\begin{equation}\label{logarithmicQualification}
\varphi_{\nu}(\lambda) = \left\{\begin{array}{ll}
\log^{-\nu}(1/\lambda) \qquad \textrm{~for} \quad 0< \lambda  \le e^{-2\nu-1},  \\
(2\nu+1)^{-\nu-0.5} \sqrt{2\nu e^{2\nu+1} \lambda + 1} \quad \textrm{~for~}  \lambda > e^{-2\nu-1},
\end{array}\right.
\end{equation}
where $\nu>0$ is a fixed number. It is not difficult to verify that $\varphi_{\nu}$ is a concave index function for all exponents $\nu>0$, and the following inequality holds for small enough $h>0$:
\begin{equation}\label{SC2}
| \varphi_{\nu}(\lambda + \zeta) - \varphi_{\nu}(\lambda) | \leq C_{\nu} \varphi_{\nu}(|\zeta|) \textrm{~for all~} \lambda\in (0, \|A\|^2], \zeta\in [-h, h],
\end{equation}
where $C_{\nu}$ is a positive constant.

If a benchmark source condition $\varphi$ is satisfies inequality (\ref{qualification1}) in Lemma \ref{Lemma_qualification1}, then other index functions $\psi$ also satisfy inequality (\ref{qualification1}) whenever they are {\it covered} by $\varphi$,
and we refer to \cite[Def.~2]{Mathe-2003} for the following definition, which is originally introduced for the {\it qualification} of a class of linear regularization methods.

\begin{definition}
\label{def:covered}
Let $\varphi(\lambda)\;(0<\lambda\le \|A^*A\|)$ be an index function. Then an index function $\psi(\lambda)\;(0<\lambda\le \|A^*A\|)$ is said to be covered by $\varphi$ if there is $\underline c>0$ such that:
 \begin{equation*}\label{QualificationCover}
\underline c\,\frac{\varphi(\alpha)}{\psi(\alpha)} \le \inf \limits _{\alpha \le \lambda \le \|A^*A\|} \frac{\varphi(\lambda)}{\psi(\lambda)} \qquad (0 < \alpha \le \bar \alpha).
\end{equation*}
\end{definition}

Similar to \cite[Prop.~3, Remark~5 and Lemma~2]{Mathe-2003}, the following proposition holds:

\begin{proposition} \label{pro:covered}
The index function $\psi$ satisfies the inequality (\ref{qualification1}) if $\psi$ is covered by $\varphi$ and if $\varphi$ satisfies the inequality (\ref{qualification1}). If the quotient function $\lambda \mapsto  \frac{\varphi(\lambda)}{\psi(\lambda)}$ is increasing for $0<\lambda \le \bar \lambda$ and some $\bar \lambda>0$, then $\psi$ is covered by $\varphi$. If, in particular, the index function $\psi(\lambda)$ is concave for
$0<\lambda \le \bar \lambda$, then $\psi$ is covered by $\varphi(\lambda)=\lambda$.
\end{proposition}

\begin{theorem}\label{reguThm_Logpriori2}
Let $(x^{h,\delta}_{k})^\infty_{k=1}\subset L^2_+(\Omega)$ be the sequence generated by Algorithm 1. Then, under the logarithmic source condition (\ref{logarithmicQualification}), if the iterative number is chosen by $k\sim (h+\delta)^{-a}$, where $a\in(0,1)$ is a fixed number, we have the estimate:
\begin{equation*}\label{prioriRate3}
\|x^{h,\delta}_{k+1}-x^\dagger\| = \mathcal{O}\left(\log^{-\nu}\left( (h+\delta)^{-1} \right) \right) \textrm{~as~} h,\delta\to 0.
\end{equation*}
\end{theorem}

\begin{proof}
By Proposition \ref{pro:covered}, a constant  $\gamma_2$ exists such that:
\begin{equation*}\label{qualification2}
\sup_{\lambda\in(0,\|A^*A\|]} \varphi_{\nu}(\lambda) g_k(\lambda) \leq \gamma_2 \varphi_{\nu}(k^{-1}).
\end{equation*}
Then, according to (\ref{reguThm_priori1ProofIneq}), cf. Theorem \ref{reguThm_priori1}, we deduce together with (\ref{SC2}) that:
\begin{equation*}
\begin{array}{ll}
& \hspace{-5mm}  \|x^{h,\delta}_{k}-x^\dagger\| \\ &
\hspace{-5mm} \leq \left\| g_k(A^*_hA_h) \varphi_{\nu}(A^*_hA_h) v \right\| + \left\| \varphi_{\nu}(A^*_hA_h) v - \varphi_{\nu}(A^*A) v \right\| + k \left( C_3 h + \frac{2}{\sqrt{\|G\|}} \delta \right)
\\ & \hspace{-5mm} \leq \gamma_2 \log^{-\nu}(k) + \|v\| \sup\limits_{\lambda\in (0, \|A\|^2],~ \zeta\in [-C_5 h, C_5 h]} | \varphi_{\nu}(\lambda + \zeta) - \varphi_{\nu}(\lambda) | +  k \left( C_3 h + \frac{2}{\sqrt{\|G\|}} \delta \right) \\ &
\hspace{-5mm} \leq \gamma_2 \log^{-\nu}(k) + \|v\| C_{\nu}  \varphi_{\nu}(C_5 h) +  k \left( C_3 h + \frac{2}{\sqrt{\|G\|}} \delta \right),
\end{array}
\end{equation*}
which yields the required estimate by combining the definition of $k$ and the fact that $\varphi_{\nu}(C_5 h)\leq h$ for small enough $h$.
\end{proof}


\section{The second non-negativity preserving iterative regularization method}

\subsection{Construction of Algorithm 2}

In this section, we do not assume the injectivity of $A_h$, and we are interested in the unique $x_0$-minimum norm solution $\xdag\in \bar X$ for arbitrary starting point $x_0\in L^2(\Omega)$. For algorithmic reasons we introduce an output map $f_+: L^2(\Omega)\to L^2_+(\Omega)$ obeying the stability inequality:
\begin{equation}\label{fIneq}
\| f_+(x) - x_{plus} \| \leq C_f \| x - x_{plus} \| \quad \textrm{~for all~} x\in L^2(\Omega),\; x_{plus} \in L^2_+(\Omega),
\end{equation}
where $C_f>0$ is a fixed constant.

\begin{example}{\rm
It is clear that $f_+(x)=a x + (1-a)|x|$ satisfies the stability condition (\ref{fIneq}) for all $a\in[0,1/2]$. If $a=1/2$, i.e. $f_+(x)=x_+=\frac{|x|+x}{2}$, besides preserving the positivity, $f_+(x)$ can enhance the sparsity of the designed approximate solution.
}\end{example}

Now we are in a position to provide the second approach based on the fixed point equation (\ref{fixed_point}).

\begin{algorithm}
\label{algorithm}
\caption{Second non-negativity preserving iterative regularization method for ill-posed operator equation (\ref{OperatorEq}) under non-negativity constraints.}
\begin{itemize}
\item[] \textbf{Input:} Imperfect forward model $A_h$ and noisy data $y^\delta$ with noise levels $h$ and $\delta$.
\item[] \textbf{Parameters}:$\;x_0\in L^2(\Omega)$, $\{\alpha_k\}\subset(0,1)$ such that:
\begin{equation}\label{alphak}
\lim_{k\to \infty} \alpha_k =0,~ \sum^\infty_{k=1} \alpha_k =\infty, \textrm{~and~}  \sum^\infty_{k=1} | \alpha_{k+1} - \alpha_k| <\infty.
\end{equation}
\item[] \textbf{Start}:$\;k\gets 1$.
\item[] \textbf{While:} $k$ does not satisfy the stopping rule,
\begin{enumerate}
\item[i.]
\begin{equation}\label{Iteration}
\qquad\qquad z^{h,\delta}_{k+1} = \alpha_k x_{0} +  (1-\alpha_k) (G + A^*_h A_h)^{-1} \left[ (G - A^*_h A_h) |z^{h,\delta}_{k} | +  2 A^*_h y^\delta \right].
\end{equation}
\item[ii.] $k \gets k+1$;
\end{enumerate}
\item[] \textbf{Output:} the obtained non-negative approximate solution $x^{h,\delta}_k := f_+(z^{h,\delta}_{k+1})$.
\end{itemize}
\end{algorithm}

We list some choices of $\alpha_k$, satisfying conditions (\ref{alphak}) in the following example.
\begin{example} {\rm $\quad$\\
(i) $\alpha_k=\frac{1}{k}$. (ii) $\alpha_k=\frac{1}{k \underbrace{\log(k)\cdots\log_q(k)}_q}$, where $\log_q(\cdot)= \max\left\{ 1, \underbrace{ \log\cdots \log }_q (\cdot) \right\}$.
}\end{example}

\begin{remark}
It is clear that our first approach in Section 3 can be viewed as a specific case of Algorithm 2 with $\alpha_k\equiv0$ and $f_+(x)=|x|$.
\end{remark}

\subsection{Convergence analysis for exact forward model and data}

Let $(z_{k})^\infty_{k=1}$ be the sequence generated by iteration (\ref{Iteration}) with $h=\delta=0$. Then, it exhibits the following properties.

\begin{lemma}\label{bounded}
For all $k\geq1$: $\|z_{k}-\bar{x}\|\leq \|x_{0}-\bar{x}\|$, where $\bar{x}\in \bar X$ is a solution of (\ref{OperatorEq}).
\end{lemma}

\begin{proof}
We prove it by induction. For $k=1$, the lemma holds since:
\begin{equation*}
\begin{array}{ll}
& \| z_{1}-\bar{x} \| = \| \alpha_1 (x_{0}-\bar{x}) + (1-\alpha_1) (G + A^* A)^{-1} (G - A^* A) ( |x_{0}| - |\bar{x}|) \| \\
& \qquad\qquad \leq \alpha_1 \| x_{0}-\bar{x}\| + (1-\alpha_1) \left\| (G + A^* A)^{-1} (G - A^* A) \right\| \|x_{0} - \bar{x}\| \\
& \qquad\qquad \leq  \alpha_1 \| x_{0}-\bar{x}\| + (1-\alpha_1) \|x_{0} - \bar{x}\| = \|x_{0}-\bar{x}\|.
\end{array}
\end{equation*}

Now, assume that $\|z_{k}-\bar{x}\|\leq \|x_{0}-\bar{x}\|$ holds for all $k\leq k_0-1$, and let's check the case where $k=k_0$. Indeed, this can be done according to the following inequalities:
\begin{equation*}
\begin{array}{ll}
& \| z_{k_0}-\bar{x}\| \leq \alpha_{k_0} \| x_{0}-\bar{x}\| + (1-\alpha_{k_0}) \left\| (G + A^* A)^{-1} (G - A^* A) ( |z_{k_0-1}| - |\bar{x}|) \right\| \\
& \qquad\qquad \leq \alpha_{k_0} \| x_{0}-\bar{x}\| + (1-\alpha_{k_0}) \left\| |z_{k_0-1}| - |\bar{x}| \right\| \\
& \qquad\qquad \leq \alpha_{k_0} \| x_{0}-\bar{x}\| + (1-\alpha_{k_0}) \left\| z_{k_0-1} - \bar{x} \right\| \leq \|x_{0}-\bar{x}\|.
\end{array}
\end{equation*}

\end{proof}

\begin{lemma}\label{CauchySequence}
$z_{k+1}-z_{k} \to 0$ as $k\to\infty$.
\end{lemma}

\begin{proof}
By using the iteration (\ref{Iteration}) and the definition of $\bar{x}\in \bar{X}$, we have:
\begin{equation*}
\begin{array}{ll}
& \hspace{-5mm} z_{k+1}-z_{k} = (\alpha_k - \alpha_{k-1})(x_{0} - \bar{x}) +  (1-\alpha_k) (G + A^* A)^{-1} (G - A^* A) (|z_{k} | - |z_{k-1} |)  \\
& \qquad\qquad + (\alpha_{k-1} - \alpha_{k})  (G + A^* A)^{-1} (G - A^* A) ( |z_{k-1} | - |\bar{x}|),
\end{array}
\end{equation*}
which implies together with Lemma \ref{bounded} that:
\begin{equation}\label{NoNoise_pfReguIneq1}
\begin{array}{ll}
& \hspace{-5mm} \|z_{k+1}-z_{k}\| \leq |\alpha_{k-1}-\alpha_{k}| \|x_0 - \bar{x} \| + (1-\alpha_k) \|z_{k} - z_{k-1}\|  + |\alpha_{k-1}-\alpha_{k}|  \|z_{k-1} - \bar{x}\|
\\ & \qquad\qquad \leq 2 |\alpha_{k-1}-\alpha_{k}| \|x_0 - |\bar{x}| \| + (1-\alpha_k) \|z_{k} - z_{k-1}\|.
\end{array}
\end{equation}
Therefore, for any $n<k$ we get:
\begin{equation}\label{pfReguIneq2}
\|z_{k+1}-z_{k}\|  \leq 2 \sum^k_{i=n} |\alpha_{i}-\alpha_{i-1}| \|x_0 - |\bar{x}| \| + \|z_{n}-z_{n-1}\| \prod^k_{i=n} (1-\alpha_{i})
\end{equation}

On the other hand, since $ \sum^k_{i=1} \alpha_i \to \infty$ as $k\to\infty$, we deduce that for any fixed $n$: $\lim_{k\to\infty} \prod^k_{i=n} (1-\alpha_{i}) =0$. This limit together with (\ref{pfReguIneq2}) yields:
\begin{equation}\label{pfReguIneq3}
\limsup_{k\to \infty} \|z_{k+1}-z_{k}\| \leq 2 \limsup_{n,k\to \infty} \sum^k_{i=n} |\alpha_{i}-\alpha_{i-1}| \|x_0-|\bar{x}| \| = 0
\end{equation}
by noting that $\sum^\infty_{k=1} | \alpha_{k+1} - \alpha_k| <\infty$.

\end{proof}

With the help of Lemmas \ref{bounded} and \ref{CauchySequence}, we are ready to show the main result in this subsection.

\begin{theorem}\label{NoiseFreeConvergence}
Let $(z_{k})^\infty_{k=1}$  be the sequence, generated by (\ref{Iteration}) with $h=\delta=0$. Then, $z_{k}\to x^\dagger$ as $k\to\infty$.
\end{theorem}

\begin{proof}
By Lemma \ref{bounded}, sequence $(z_k)$ is uniformly bounded. Then, according to Lemma \ref{CauchySequence} and the following inequalities,
\begin{equation*}
\begin{array}{ll}
& \hspace{-8mm}  \left\| z_{k} - (G + A^* A)^{-1} \left[ (G - A^* A) |z_k| + 2 A^* y \right] \right\| \\
& \leq \left\| z_{k} - (1-\alpha_k) (G + A^* A)^{-1} \left[ (G - A^* A) |z_{k-1}| + 2 A^* y \right] \right\| \\
& \qquad\qquad + (1-\alpha_k) \Big\| (G + A^* A)^{-1} \left[ (G - A^* A) |z_{k-1}| + 2 A^* y \right] \\
& \qquad\qquad\qquad - (G + A^* A)^{-1} \left[ (G - A^* A) |z_{k}| + 2 A^* y \right] \Big\| \\
& \qquad\qquad + \alpha_k \left\| (G + A^* A)^{-1} \left[ (G - A^* A) |z_k| + 2 A^* y \right] \right\|  \\
&  \leq \alpha_k \|x_{0}\| + (1-\alpha_k) \|z_{k-1}-z_{k}\| + \alpha_k \left( \|z_k\| + 2\|(G + A^* A)^{-1}A^* y\| \right),
\end{array}
\end{equation*}
we deduce together with the property of $\alpha_k$ that,
\begin{equation}\label{limitSup1}
\lim_{k\to \infty}\left\| z_{k} - (G + A^* A)^{-1} \left[ (G - A^* A) |z_k| + 2 A^* y \right] \right\| = 0.
\end{equation}

We show next that,
\begin{equation}\label{limitSup2}
\limsup_{k\to \infty} \langle (G + A^* A)^{-1} \left[ (G - A^* A) |z_k| + 2 A^* y \right] - x^\dagger, x_0 - x^\dagger  \rangle \leq 0.
\end{equation}
We prove this by contradiction. Assume that there exist a subsequence $(k_j)$ and a number $d>0$ such that  for all $j\in \mathbb{N}$,
\begin{equation}\label{PfIneq1}
\langle (G + A^* A)^{-1} \left[ (G - A^* A) |z_{k_j}| + 2 A^* y \right] - x^\dagger, x_0 - x^\dagger  \rangle \geq d.
\end{equation}
Since $(z_{k_j})$ is uniformly bounded, without loss of generality, we may also assume that $(z_{k_j})$ is weakly convergent to an element, say, $\bar{z}$. By (\ref{limitSup1}), the sequence $((G + A^* A)^{-1} \left[ (G - A^* A) |z_{k_j}| + 2 A^* y \right])$ also converges weakly to $\bar{z}$. Hence, $\bar{z}$ is a solution of (\ref{identity_x}).

On the other hand, by the definition of $x^\dagger$, i.e. $x^\dagger=\arg\min_{x \in \bar X} \|x-x_0\|^2$, for any solution of (\ref{identity_x}), say $\bar{z}$, we have the following variational inequality,
\begin{equation}\label{PfIneq2}
\langle x_0 - x^\dagger, \bar{z} - x^\dagger \rangle \leq 0.
\end{equation}

Combining (\ref{PfIneq1}) and (\ref{PfIneq2}) we get the contradiction:
\begin{equation*}
\begin{array}{ll}
& d \leq \limsup_{j\to \infty} \langle (G + A^* A)^{-1} \left[ (G - A^* A) |z_{k_j}| + 2 A^* y \right] - x^\dagger, x_0 - x^\dagger  \rangle \\
& \qquad = \langle \bar{z} - x^\dagger, x_0 - x^\dagger \rangle \leq 0.
\end{array}
\end{equation*}

By (\ref{limitSup2}) and the definition of $\alpha_k$, for any fixed $\epsilon>0$, there exists a number $k'=k'(\epsilon)$ such that for all $k\geq k'$, the following two inequalities holds simultaneously:
\begin{equation*}
\hspace{-10mm} \langle (G + A^* A)^{-1} \left[ (G - A^* A) |z_k| + 2 A^* y \right] - x^\dagger, x_0 - x^\dagger  \rangle \leq \epsilon, ~ \alpha_k \|x_0- x^\dagger\|^2\leq \epsilon.
\end{equation*}
Then, for all $k\geq k'+1$ we have:
\begin{equation*}
\begin{array}{ll}
& \hspace{-7mm} \|z_{k} - x^\dagger \|^2 = \alpha^2_k \|x_{0} - x^\dagger \|^2 \\
&  + 2 \alpha_k (1-\alpha_k) \langle (G + A^* A)^{-1} \left[ (G - A^* A) |z_{k-1}| + 2 A^* y \right] - x^\dagger, x_0 - x^\dagger  \rangle \\
&  + (1-\alpha_k)^2 \|(G + A^* A)^{-1} \left[ (G - A^* A) |z_{k-1}| + 2 A^* y \right] - x^\dagger \|^2 \\
& \hspace{-5mm} \leq \alpha^2_k \|x_{0} - x^\dagger \|^2 + 2  \alpha_k (1-\alpha_k) \epsilon  \\
& \qquad + (1-\alpha_k)^2 \|(G + A^* A)^{-1} \left[ (G - A^* A) |z_{k-1}| + 2 A^* y \right] - x^\dagger \|^2 \\
& \hspace{-5mm} \leq 3 \alpha_k \epsilon + (1-\alpha_k)^2 \| (G + A^* A)^{-1} \left[ (G - A^* A) |z_{k-1}| + 2 A^* y \right] \\
& \qquad\qquad\qquad\qquad\qquad - (G + A^* A)^{-1} \left[ (G - A^* A) x^\dagger+ 2 A^* y \right] \|^2
\end{array}
\end{equation*}
\begin{equation*}
\begin{array}{ll}
& \hspace{-5mm} \leq 3 \alpha_k \epsilon + (1-\alpha_k) \| |z_{k-1}| - x^\dagger \|^2 \leq 3 \alpha_k \epsilon + (1-\alpha_k) \| z_{k-1} - x^\dagger \|^2,
\end{array}
\end{equation*}
which implies that:
\begin{equation}\label{PfIneq3}
\|z_{k} - x^\dagger \|^2 - 3 \epsilon \leq (1-\alpha_k) \left( \|z_{k-1} - x^\dagger \|^2  - 3 \epsilon \right)  \leq  \prod^k_{i=1} (1-\alpha_i) \left( \|x_{0} - x^\dagger \|^2  - 3 \epsilon \right).
\end{equation}

Since $ \sum^\infty_{i=1} \alpha_i =\infty$ implies $\prod^\infty_{i=1} (1-\alpha_{i}) =0$, we obtain together with (\ref{PfIneq3}) that:
\begin{equation}\label{limitSup3}
\limsup_{k\to \infty} \|z_{k} - x^\dagger \|^2 \leq 3 \epsilon,
\end{equation}
which yields the required result as $\epsilon>0$ can be made arbitrarily small.

\end{proof}

\begin{remark}
By the definition of $f_+$ in (\ref{fIneq}), we have immediately the strong convergence of Algorithm 2 for exact operator $A$ and right-hand side $y$, i.e. $x_k\to x^\dagger$ as $k\to\infty$.
\end{remark}


\subsection{Regularization property of Algorithm 2}

In this section, we prove the convergence of Algorithm 2 with respect to the noise levels under both a priori and a posteriori stopping rules.

\begin{proposition}\label{estimateNoisy}
Let $(z^{h,\delta}_{k})^\infty_{k=1}$ and $(z_{k})^\infty_{k=1}$ be two sequences, generated by iteration (\ref{Iteration}) with noisy information $\{A_h, y^\delta\}$ and exact information $\{A, y\}$, respectively. Then, there exists a constant $C_7$ such that:
\begin{equation}\label{lemma_operator_ineq2}
\|z^{h,\delta}_{k}  - z_{k} \| \leq C_7 (h+\delta) \sum^k_{j=1} \prod^{k}_{i=j} (1-\alpha_i) .
\end{equation}
\end{proposition}

\begin{proof}
By Lemma \ref{bounded}, we have $\|z_{k}\|\leq \|\bar{x}\| + \|x_{0}-\bar{x}\|$ for all $k\geq1$, which implies together with Lemma \ref{lemma_operator} that:
\begin{equation*}
\begin{array}{ll}
& \|z^{h,\delta}_{k}-z_{k}\| = (1-\alpha_k) \Big\|  (G  +  A^*_hA_h)^{-1} \left[ (G - A^*_hA_h) |z^{h,\delta}_{k-1} | +  2 A^*_h y^\delta \right] \\
& \qquad \quad - (G  +  A^*A)^{-1} \left[ (G - A^*A) |z_{k-1} | +  2 A^* y \right] \Big\| \\
& \leq (1-\alpha_k) \Big\{  \|(G  +  A^*_hA_h)^{-1} (G - A^*_hA_h) (|z^{h,\delta}_{k-1} | - |z_{k-1}| )\| \\
& \qquad \quad + \| [(G  +  A^*_hA_h)^{-1} (G - A^*_hA_h) - (G  +  A^*A)^{-1} (G - A^*A)] |z_{k-1}| \|  \\
& \qquad \quad + 2 \|(G  +  A^*_hA_h)^{-1} A^*_h y^\delta - (G  +  A^*_hA_h)^{-1} A^*_h y \|
\\
& \qquad \quad + 2 \|(G  +  A^*_hA_h)^{-1} A^*_h y - (G  +  A^*A)^{-1} A^* y \| \Big \} \\
& \leq (1-\alpha_k) \Big\{ \|z^{h,\delta}_{k-1}  - z_{k-1}\| + C_1 (\|\bar{x}\| + \|x_{0}-\bar{x}\|) h + \frac{4}{\sqrt{\|G\|}} \delta + 2 C_2 \|y\| h \Big \}.
\end{array}
\end{equation*}
We complete the proof by noting $z^{h,\delta}_{0}  = z_{0} = x_0$ and defining that $C_7=\max\{ C_1 (\|\bar{x}\|$ $+ \|x_{0}-\bar{x}\|) + 2 C_2\|y\|, 4/\sqrt{\|G\|} \}$.
\end{proof}

By combining Theorem \ref{NoiseFreeConvergence}, Proposition \ref{estimateNoisy}, and the inequality $\sum^k_{j=1} \prod^{k}_{i=j} (1-\alpha_i)\leq k$, we obtain immediately the regularization property of Algorithm 2 as follows.

\begin{theorem}\label{reguThm2}
Let $(x^{h,\delta}_{k})^\infty_{k=1} \subset L^2_+(\Omega)$ be the sequence generated by Algorithm 2 with general type of $G$, defined in (\ref{eq:multiop}). Then, if the iterative number $k^*=k^*(h,\delta)$ is chosen such that:
\begin{equation*}\label{k_condition2}
k^*\to \infty \textrm{~and~} (h+\delta)k^*  \to 0 \textrm{~as~} h,\delta \to 0,
\end{equation*}
the approximate solution $x^{h,\delta}_{k^*}$ converges to $x^\dagger$ as $\delta,h\to0$.
\end{theorem}

Similar to Theorem \ref{reguThmPosteriori}, Algorithm 2 yields a regularization scheme under the modified discrepancy principle (\ref{Method1discrepancy}), i.e. the following theorem holds.

\begin{theorem}\label{ThmRegu1}
Assume that $x_0=0$. Let $x^{h,\delta}_{k^*} \in L^2_+(\Omega)$ be the approximate solution, obtained by Algorithm 2 with $G=\mu I$ and the stopping rule (\ref{Method1discrepancy}) and $\alpha_k$ satisfying both conditions (\ref{alphak}) and \footnote[5]{It is not difficult to verify that $\alpha_k = 1/k$ satisfies condition (\ref{AssumptionAlpha}).}
\begin{equation}\label{AssumptionAlpha}
\limsup_{k\to \infty} \sum^k_{i=1} \alpha_i \prod^k_{j=i} (1-\alpha_j) < \infty.
\end{equation}
Then,
\begin{itemize}
\item[(i)] if $f_+(\cdot)=|\cdot|$, the approximate solution $x^{h,\delta}_{k^*}$ converges, up to a subsequence,  to the unique minimum norm solution of (\ref{OperatorEq}) as $h,\delta \to0$.

\item[(ii)] In the case $k^*\to \infty$ as $h,\delta\to 0$, assertion (i) holds for arbitrary $f_+$ satisfying (\ref{fIneq}).
\end{itemize}
\end{theorem}

\begin{proof}
According to the proof of Theorem \ref{reguThmPosteriori}, we only need to show assertion (ii). By (\ref{Iteration}) and the assumption $x_0=0$, we have:
\begin{equation}\label{PfDiff}
z^{h,\delta}_{k+1}- (1-\alpha_k)|z^{h,\delta}_{k}| = 2(1-\alpha_k) (G + A^*_h A_h)^{-1}  A^*_h( y^\delta - A_h |z^{h,\delta}_{k}|).
\end{equation}

By using (\ref{PfDiff}), we obtain together with the definition of $\tilde{G}$ in (\ref{DefGh}) and the positivity of $x^\dagger$ that:
\begin{equation*}
\begin{array}{ll}
& \|z^{h,\delta}_{k+1}- (1-\alpha_k)x^\dagger\|^2 - (1-\alpha_k)\|z^{h,\delta}_{k}- x^\dagger\|^2 \\
&   \leq \|z^{h,\delta}_{k+1}- (1-\alpha_k)x^\dagger\|^2 - \|(1-\alpha_k)|z^{h,\delta}_{k}|- (1-\alpha_k)x^\dagger\|^2 \\
&  = 2((1-\alpha_k) |z^{h,\delta}_{k}|-(1-\alpha_k)x^\dagger, z^{h,\delta}_{k+1}-(1-\alpha_k) |z^{h,\delta}_{k}|) + \|z^{h,\delta}_{k+1}-(1-\alpha_k) |z^{h,\delta}_{k}|\|^2 \\
&  = 4(1-\alpha_k)^2 \left( |z^{h,\delta}_{k}|-x^\dagger, (G + A^*_h A_h)^{-1}  A^*_h( y^\delta - A_h |z^{h,\delta}_{k}|) \right) \\
& \qquad + 4(1-\alpha_k)^2 \|(G + A^*_h A_h)^{-1}  A^*_h( y^\delta - A_h |z^{h,\delta}_{k}|)\|^2 \\
&  = 4(1-\alpha_k)^2 \left( A_h (|z^{h,\delta}_{k}|-x^\dagger) , (\tilde{G} + A_h A^*_h)^{-1} (y^\delta - A_h |z^{h,\delta}_{k}| ) \right) \\
& \qquad + 4(1-\alpha_k)^2 \|(G + A^*_h A_h)^{-1}  A^*_h( y^\delta - A_h |z^{h,\delta}_{k}|)\|^2  \\
&  =  4(1-\alpha_k)^2 \left( (\tilde{G} + A_h A^*_h)^{-1} (y^\delta - A_h |z^{h,\delta}_{k}| ) ,  y^\delta -A_h x^\dagger \right) \\
&  - 4(1-\alpha_k)^2 \left( (\tilde{G} + A_h A^*_h)^{-1} (y^\delta - A_h |z^{h,\delta}_{k}| ) , \tilde{G} (\tilde{G} + A_h A^*_h)^{-1} (y^\delta - A_h |z^{h,\delta}_{k}| \right) \\
& \leq 4(1-\alpha_k)^2 \|(\tilde{G} + A_h A^*_h)^{-1} ( y^\delta - A_h |z^{h,\delta}_{k}|)\|_{\mathcal{Y}} \\
& \qquad \cdot \left\{\delta + h C^\dagger- \mu \|\tilde{G}^{1/2}_h(\tilde{G} + A_h A^*_h)^{-1} ( y^\delta - A_h |z^{h,\delta}_{k}|)\|_{\mathcal{Y}} \right\},
\end{array}
\end{equation*}
which implies that for all $k< k^*$,
\begin{equation}\label{DecreasingPf}
\|z^{h,\delta}_{k+1}-x^\dagger\| \leq \alpha_k \|x^\dagger\| + \|z^{h,\delta}_{k+1}-(1-\alpha_k)x^\dagger\| \leq \alpha_k \|x^\dagger\| + (1-\alpha_k) \|z^{h,\delta}_{k}-x^\dagger\|
\end{equation}
by noting the choice of $k^*$ in (\ref{Method1discrepancy}). Let $(h_n, \delta_n, A_n, y^n, k_n)_n$ be the sequence as defined in the proof of Theorem \ref{reguThmPosteriori}. Then, for $n>m$,
\begin{equation}\label{PfIneqEndMethod2}
\begin{array}{ll}
& \hspace{-5mm} \|x^{h_n,\delta_n}_{k_n}-x^\dagger\| \leq C_f \|z^{h_n,\delta_n}_{k_n}-x^\dagger\| \leq C_f \left\{ \alpha_{k_n-1} \|x^\dagger\| + (1-\alpha_{k_n-1}) \|z^{h_n,\delta_n}_{k_n-1}-x^\dagger\| \right\} \\
& \hspace{-5mm} \leq C_f \left\{ \alpha_{k_n-1} \|x^\dagger\| + (1-\alpha_{k_n-1})\alpha_{k_n-2} \|x^\dagger\| +(1-\alpha_{k_n-1})(1-\alpha_{k_n-2}) \|z^{h_n,\delta_n}_{k_n-2}-x^\dagger\| \right\}  \\
& \hspace{-5mm} \leq \cdots \leq C_f \left\{  \left[ \sum^{k_n-1}_{i=k_m} \alpha_i \prod^{k_n-2}_{j=i} (1-\alpha_j) \right]\|x^\dagger\| +  \left[ \prod^{k_n-1}_{j=k_m} (1-\alpha_j)\right] \|z^{h_n,\delta_n}_{k_m}-x^\dagger\| \right\} \\
& \hspace{-5mm} \leq C_f \left\{  \left[ \sum^{k_n-1}_{i=k_m} \alpha_i \prod^{k_n-2}_{j=i} (1-\alpha_j) \right] \|x^\dagger\| + \|z^{h_n,\delta_n}_{k_m}-x^\dagger\| \right\} \\
& \hspace{-5mm} \leq C_f \left\{  \left[ \sum^{k_n-1}_{i=k_m} \alpha_i \prod^{k_n-2}_{j=i} (1-\alpha_j) \right] \|x^\dagger\| + \|z_{k_m}-x^\dagger\| + \|z_{k_m}-z^{h_n,\delta_n}_{k_m}\| \right\}.
\end{array}
\end{equation}
According to the requirement (\ref{AssumptionAlpha}) of $\alpha_k$ and  Theorem \ref{NoiseFreeConvergence}, one can fix $m$ so large that the first two terms on the right-hand side of (\ref{PfIneqEndMethod2}) are sufficiently close to zero; now that $k_m$ is fixed, we can apply inequality (\ref{lemma_operator_ineq2}), cf. Proposition \ref{estimateNoisy}, to conclude that the left-hand side of (\ref{PfIneqEndMethod2}) must go to zero when $n\to\infty$.
\end{proof}


\section{Application in biosensor tomography}

Over the last few decades, biosensor-based techniques have made a significant impact in many fields, such as antibody engineering, virology, immunology, and the pharmaceutical industry. To design an efficient biosensor instrument that is specifically functionalized, scientists must know the physical chemistry of biomolecular surface interactions. The study of interaction information from biosensor data is called \emph{Biosensor Tomography}. In our experiments, we collect biosensor data that measures the analyte biomolecules of several different concentrations on a sensor chip with immobilized ligand molecules that form complexes with analytes. Such measured data is termed as a sensorgram, where the systems response, proportional to total complex concentration, is measured over time for different analyte injections. Recently, a mathematical model, named \emph{Rate Constant Map Theory}, has been developed in \cite{AIDA} to describe the relationship between biosensor data and the rate constants distribution (RCD). In this theory, the number of local peaks of RCD and their positions are interpreted as the interaction numbers and the corresponding rate constants. However, RCD cannot be measured directly, and from the rate constant map theory, reconstructing the RCD from biosensor data is an imaging tomography problem, which is a typical ill-posed inverse problem.

To be more precise, let's recall the kinetics for biosenors. For a single kinetic model, the binding of the analyte to the ligand on the sensor chip consists two steps. First, the analyte is transferred out of the bulk solution towards the sensor chip surface. Second, the binding of the analyte to the ligand takes place. Since the second step is dominant in the binding process, we can only consider the equilibrium:
\begin{equation*}
[\mathbf{D}] +[\mathbf{L}] \autoleftrightharpoons{$k_a$}{$k_d$} [\mathbf{DL}],
\label{equilibrium}
\end{equation*}
by assuming that the mass transfer kinetics are extremely fast. Here, $[\mathbf{D}]$, $[\mathbf{L}]$, and $[\mathbf{DL}]$ denote the concentrations of the analyte, ligand, and complex, respectively. $k_a$ and $k_d$ represent the association and dissociation rate constants. Then, the rate of complex formation can be described by:
\begin{equation}
\frac{d [\mathbf{DL}](t)}{dt} = k_a \cdot [\mathbf{D}](t) \cdot [\mathbf{L}](t) - k_d \cdot [\mathbf{DL}](t),
\label{rateEq}
\end{equation}
where $t$ is the analysis time. Suppose that the analyte $\mathbf{D}$ is injected and flushed over the surface in such a way that the concentration can be assumed to be fixed during the study, i.e. $[\mathbf{D}](t)\equiv C$ ($C>0$ is a constant). Then, the amount of free ligand will decrease with time according to $[\mathbf{L}](t)=[\mathbf{L}](0)-[\mathbf{DL}](t)$. Assume that the sensor response $R$ is proportional to the complex concentration $[\mathbf{DL}](t)$, i.e. $R(t)=\sigma \cdot [\mathbf{DL}](t)$, where $\sigma$ is a positive number. Then, Equation (\ref{rateEq}) becomes:
\begin{equation}
\frac{d R(t)}{dt} = k_a \cdot C \cdot \left( R_{max} - R(t) \right) - k_d \cdot R(t),
\label{rateEq3}
\end{equation}
where $R_{max}=\sigma \cdot[\mathbf{L}](0)$. Set $R(t_0)=0$, by solving (\ref{rateEq3}) we obtain that:
\begin{equation}
R(t) = R_{max} \cdot \frac{k_a C}{k_d + k_a C} \cdot\left( 1- e^{(k_d + k_a C)(t-t_0)} \right).
\label{rateEq4}
\end{equation}

Now, suppose that in the whole chemistry reaction there are $m$ analytes and $n$ binding sites on the biosensor surface and first order kinetics. Denote by $(k_{a,i}, k_{d,j})$ the pair of association and dissociation constants for the interaction between the $i$th analyte and $j$th binding site. Let $R_{i,j}(t)$ be the response at time $t$ of a complex with association constant $k_{a,i}$ and dissociation constant $k_{d,j}$. Then, according to (\ref{rateEq4}) we have:

\begin{equation}\label{Rij}
\begin{array}{ll}
& \hspace{-5mm} R_{i,j}(t) = \\ &
\left\{
\begin{array}{ll}
& 0, \qquad t\leq t_0 + \Delta t, \\
& R^{max}_{i,j}  \frac{k_{a,i} C}{k_{d,j}+ k_{a,i} C} \left( 1- e^{-(k_{d,j}+ k_{a,i} C)(t-t_0)} \right), ~t_0+ \Delta t < t \leq t_0+ t_{inj}+ \Delta t, \\ &
R^{max}_{i,j} \frac{k_{a,i} C}{k_{d,j}+ k_{a,i} C} \left( 1- e^{-(k_{d,j}+ k_{a,i} C) t_{inj}} \right) e^{-k_{d,j}(t-t_0-t_{inj})} ,   t> t_0+ t_{inj}+ \Delta t,
\end{array}\right.
\end{array}
\end{equation}
where constant $C$ is the concentration of the analyte, $t_0$ is the time when the injection of the analyte begins, and $t_{inj}$ is the injection time. The adjustment parameter $\Delta t$ is a time delay that accounts for the fact that it usually takes some time for the detector to respond to the injection. Constant $R^{max}_{i,j}$ is the total surface binding capacity, corresponding to association and dissociation constants $k_{a,i}$ and $k_{d,j}$, i.e., the detector response when every binding site on the biosensor surface has formed a complex with the analyte.

In \emph{Rate Constant Map Theory}, we assume that the total measured response, $y$, of a system can be written as a linear combination of some individual responses $R_{i,j}$, namely $y  = \sum^{m,n}_{i,j=1} R_{i,j}$. If we let $m,n\to+\infty$ in the above equation, we obtain the final continuous model of the biosensor tomography,
\begin{equation}
A x:= \int_{\Omega} K(t,C;k_a,k_d) x(k_a,k_d) d k_a d k_d = y(t;C), \quad (t,C)\in \Theta,
\label{IntegralEq}
\end{equation}
where $\Omega, \Theta\subset \mathbb{R}^2_+$ denote the interested domain of rate constants and the measurement domain for variables $t$ and $C$, respectively, and the kernel function $K(\cdot)$ is defined as:
\begin{equation}\label{Rij}
\begin{array}{ll}
& \hspace{-5mm} K(t,C;k_a,k_d) = \\
&  \left\{
\begin{array}{ll}
& 0, \qquad t\leq t_0 + \Delta t, \\
& \frac{k_{a} C}{k_{d}+ k_{a} C} \left( 1- e^{-(k_{d}+ k_{a} C)(t-t_0)} \right), ~t_0+ \Delta t < t \leq t_0+ t_{inj}+ \Delta t, \\ &
\frac{k_{a} C}{k_{d}+ k_{a} C} \left( 1- e^{-(k_{d}+ k_{a} C) t_{inj}} \right) e^{-k_{d}(t-t_0-t_{inj})}, ~t> t_0+ t_{inj}+ \Delta t.
\end{array}\right.
\end{array}
\end{equation}
Here the function $x(k_a,k_d)$, which is the generalization of the total surface binding capacity $\{R^{max}_{i,j}\}$ in our model, is known as the (continuous) rate constant map; see \cite{Svitel-2003} for details.

In practice, besides the noise in the measurement data $y^\delta$, the forward model is also known inexactly. The main uncertainty in our model (\ref{IntegralEq}) is the adjustment parameter $\Delta t$, which is usually estimated by experience according to the shape of the response curve. Let $K_h$ be the inexact kernel function of $K$ with the adjustment parameter $\Delta t$ replaced by a perturbed one $\Delta t_h$. Then, for the inexact forward operator $A_h$ with perturbed kernel function $K_h$, we have:

\begin{lemma}\label{Ah}
If $|\Delta t_h - \Delta t|\leq h/\sqrt{2}$, then $\|A_h - A\|_{L^2(\Omega) \to L^2(\Theta)} \leq  h$.
\end{lemma}

\begin{proof}
Define by $H= K_h - K$ the difference between exact and inexact kernel functions. Then, we have:

\begin{equation}\label{H}
H = \left\{
\begin{array}{ll}
& 0, \qquad t\leq t_0 + \min(\Delta t, \Delta t_h), \\
& \frac{k_{a} C}{k_{d}+ k_{a} C} \left( 1- e^{-(k_{d}+ k_{a} C)(t-t_0)} \right), \\
& \qquad t_0 + \min(\Delta t, \Delta t_h) < t \leq t_0 + \max(\Delta t, \Delta t_h), \\
& 0, \qquad  t_0 + \max(\Delta t, \Delta t_h) < t \leq t_0+ t_{inj}+ \min(\Delta t, \Delta t_h), \\ &
\frac{k_{a} C}{k_{d}+ k_{a} C} \left( 1- e^{-(k_{d}+ k_{a} C) t_{inj}} \right) e^{-k_{d}(t-t_0-t_{inj})} , \\
& \qquad  t_0+ t_{inj}+ \min(\Delta t, \Delta t_h) < t \leq t_0+ t_{inj}+ \max(\Delta t, \Delta t_h), \\
& 0, \qquad\qquad  t> t_0+ t_{inj}+ \max(\Delta t, \Delta t_h),
\end{array}\right.
\end{equation}
which implies together with the Cauchy-Schwarz inequality that:
\begin{equation*}
\begin{array}{ll}
& \|A_h - A\|_{L^2(\Omega) \to L^2(\Theta)} = \max\limits_{x: \|x\|_{L^2(\Omega)}=1} \|A_h x - Ax\|_{L^2(\Theta)}  \\
& \qquad =  \max\limits_{x: \|x\|_{L^2(\Omega)}=1} \|\int_{\Omega} H(t,C;k_a,k_d) x(k_a,k_d) d k_a d k_d\|_{L^2(\Theta)} \\
& \qquad  \leq \sqrt{\int_{\Omega} H^2(t,C;k_a,k_d) d k_a d k_d} \leq h,
\end{array}
\end{equation*}
which yields the required inequality.
\end{proof}

In order to apply Algorithms 1 and 2 for solving (\ref{IntegralEq}), at the end of this section, we derive the explicit formulas for quantities $A^*_h A_h x$ and $A^*_h y^\delta$ for our 2D integral operator as follows:
\begin{equation*}
\label{AAx}
A^*_h A_h x= \int_{\Omega} x(k'_a,k'_d) \left[ \int_{\Theta} K_h(t,C;k'_a,k'_d) K_h(t,C;k_a,k_d) dt dC  \right] dk'_a dk'_d,
\end{equation*}
\begin{equation*}
\label{Ay}
A^*_h y^\delta= \int_{\Theta} K_h(t,C;k_a,k_d) y^\delta(t,C) dt dC.
\end{equation*}


\section{Computer simulations}
\label{Simulations}

\subsection{Tests for model problems}

In this section, we present some artificial examples to demonstrate the effectiveness of
the proposed non-negativity preserving iterative regularization methods~-- Algorithms 1 and 2. We take the biosensor tomography considered in Section 5 as an example. The simulation
consists of two steps. First, a simulated signal $y$ (input signal) is generated by computer according to equation (\ref{IntegralEq}) for a given solution $x(k_a,k_d)$ (input rate constant map). Then, the polluted data $y^\delta$, which is generated by adding the artificial noise, is processed through our algorithms, and the retrieved rate constant map is compared with input map. To this end, we divide both sides of integral equation (\ref{IntegralEq}) by the constant $2\sqrt{\int_{\Theta} \int_{\Omega} K^2(t,C;k_a,k_d) dk_a dk_d dt dC}$. Then, the newly defined integral operator, denoted also by $A$, satisfies $\|A\|_{L^2(\Omega)\to L^2(\Theta)} \leq 1/2$. If the inexact adjustment parameter $\Delta t_h$ fulfills the inequality $|\Delta t_h - \Delta t|\leq h'$ with $h'\leq 1/\sqrt{8}$, the corresponding noisy operator $A_h$ satisfies $\|A_h\|_{L^2(\Omega)\to L^2(\Theta)} \leq \|A\|_{L^2(\Omega)\to L^2(\Theta)} +\sqrt{2} h'\leq1$.

The main parts in Algorithms 1 and 2 are iterations in (\ref{Iteration2}) and (\ref{Iteration}), which are updated by solving the following equations,
\begin{equation}
\label{Eq2Iteration}
(G + A^*_h A_h) z^{h,\delta}_{k+1} = (G - A^*_h A_h) |z^{h,\delta}_{k} | +  2 A^*_h y^\delta
\end{equation}
and
\begin{equation}
\label{Eq1Iteration}
\hspace{-2mm} (G + A^*_h A_h) z^{h,\delta}_{k+1} = \alpha_k (G + A^*_h A_h) x_{0} +  (1-\alpha_k) \left[ (G - A^*_h A_h) |z^{h,\delta}_{k} | +  2 A^*_h y^\delta \right].
\end{equation}

It should be noted that for $G=\mu I$, (\ref{Eq2Iteration}) and (\ref{Eq1Iteration}) are actually the Fredholm integral equations of the second kind, which can be solved by various numerical approaches, see e.g. \cite[Chapter 12]{AtkinsonHan2009}. In this work, we apply linear finite elements to solve equations (\ref{Eq2Iteration}) and (\ref{Eq1Iteration}). To this end, let $\mathcal{Y}_n$ be the $n$ dimensional approximation of $\mathcal{Y}$ with an orthonormal basis $\{\phi_i(t,C)\}^n_{i=1} \subset L^2_+(\Omega)$. Denote by $P_n$ the orthogonal projection operator acting from $\mathcal{Y}$ into $\mathcal{Y}_n$. Define $A_n:= P_n A$ and $\mathcal{X}_n:= A^*_n \mathcal{Y}_n \subset L^2(\Omega)$. Denote by $\psi_j(k_a,k_d)=\int_{\Theta} K(t,C;k_a,k_d) \phi_j(t,C) dt dC \in L^2_+(\Omega)$. Then, the finite element solution $\hat{x}_n(k_a,k_d)\in \mathcal{X}_n$ and the finite approximation of data $\hat{y}_n(t,C) \in \mathcal{Y}_n$ have the decompositions $\hat{x}_n(k_a,k_d)=\sum^n_{j=1} [\mathbf{x}]_j \, \psi_j(k_a,k_d)$ and $\hat{y}_n(t,C) =\sum^n_{i=1} [\mathbf{y}]_i \, \phi_i(t,C)$ with the coefficient vectors $\mathbf{x}$ and $\mathbf{y}$, respectively. Consequently, the finite element approximation of integral equation (\ref{IntegralEq}) can be written as the following system of linear algebraic equations,
\begin{equation}\label{MatrixEq}
\mathbf{A} \mathbf{x} = \mathbf{y},
\end{equation}
where
\begin{equation}\label{Amatrix}
[\mathbf{A}]_{ij} = \int_{\Omega} \psi_j(k_a,k_d) \left\{ \int_{\Theta} K(t,C;k_a,k_d) \phi_i(t,C) dt dC \right\}  dk_a dk_d.
\end{equation}

Uniformly distributed noises with the magnitudes $h'\in(0,1/\sqrt{8})$ and $\delta'>0$ are added to the accuracy adjustment parameter and discretized exact right-hand side:
\begin{equation*}\label{Data}
\begin{array}{ll}
& \Delta t_h := \left[ 1 + h' \cdot(2 \textrm{Rand}(\cdot) -1) \right] \cdot \Delta t, \\
&
[\mathbf{y}^\delta]_i := \left[ 1 + \delta' \cdot(2 \textrm{Rand}(\cdot) -1) \right] \cdot [\mathbf{y}]_i, \quad i=1, ..., n,
\end{array}
\end{equation*}
where $\textrm{Rand}(\cdot)$ returns a pseudo-random value drawn from a uniform distribution on [0,1]. The noise level of perturbed operator and  measurement data are simply calculated by $h=\|\mathbf{A}_h - \mathbf{A}\|_2$ and $\delta=\|\mathbf{y}^\delta - \mathbf{y}\|_2$, respectively, where $\|\cdot\|_2$ denotes the standard matrix or vector norm in Euclidean space. Here, $\mathbf{A}_h$ is defined in (\ref{Amatrix}) with $\Delta t$ in $K(t,C;k_a,k_d)$ replaced by $\Delta t_h$.

Now, let's consider the numerical solution of equations (\ref{Eq2Iteration}) and (\ref{Eq1Iteration}). For clarity of statement, we only consider the equation (\ref{Eq2Iteration}), whose finite element solution can be obtained by solving the following system of linear equations,
\begin{equation}\label{FiniteModel}
\mathbf{G} \mathbf{z}^{k+1} + \mathbf{B} \mathbf{z}^{k+1} = \mathbf{f}(\mathbf{z}^{k}),
\end{equation}
where
\begin{equation*}
\begin{array}{ll}
& [\mathbf{B}]_{ij}= \int_{\Omega} g_j(k_a,k_d) \psi_i(k_a,k_d) dk_a dk_d, \\
& [\mathbf{f}(\mathbf{z}^{k})]_i = 2 \int_{\Omega} \psi_i(k_a,k_d) \left\{ \int_{\Theta} K_h(t,C;k_a,k_d) \phi_i(t,C) dt dC \right\}  dk_a dk_d \cdot [\mathbf{y}^\delta]_i \\
& \qquad\qquad + [\mathbf{G} \mathbf{z}^{k}]_i -  \left| [\mathbf{z}^{k}]_i \right| \left\{ \int_{\Omega} g_i(k_a,k_d) \psi_i(k_a,k_d) dk_a dk_d \right\}, \\
& g_j(k_a,k_d)= \int_{\Omega} \psi_j(k'_a,k'_d) \left[ \int_{\Theta} K_h(t,C;k'_a,k'_d) K_h(t,C;k_a,k_d) dt dC  \right] dk'_a dk'_d.
\end{array}
\end{equation*}
Here $\mathbf{G}$ is a positive-definite matrix of size $n\times n$.

In this section, we investigate our two approaches with the general type of $G$, which is assumed to be a strictly positive definite and self-adjoint bounded linear operator. It is not difficult to show that if $G$ has the type of (\ref{eq:multiop}), and there exists an associated invertible operator $\tilde{G}_h: \mathcal{Y} \to \mathcal{Y}$ of $G$ such that,
\begin{equation*}
(G + A^*_h A_h)^{-1} A^*_h = A^*_h (\tilde{G}_h + A_h A^*_h)^{-1} \textrm{~and~} \tilde{G}_h=(\tilde{G}^{1/2}_h)^*\tilde{G}^{1/2}_h,
\end{equation*}
both Theorems \ref{reguThmPosteriori} and \ref{ThmRegu1} still hold under the following modified discrepancy principle: find $k^*$ such that for all $k \leq  k^*$,
\begin{equation}\label{Method1discrepancy2}
\begin{array}{ll}
& \|\tilde{G}^{1/2}_h\| \cdot \|\tilde{G}^{1/2}_h(\tilde{G}_h + A_h A^*_h)^{-1} ( y^\delta - A_h |z^{h,\delta}_{k^*}|)\|_{\mathcal{Y}}   \\
&  \qquad \leq \tau_0 (\delta + h C^\dagger)  \leq  \|\tilde{G}^{1/2}_h\|  \cdot \|\tilde{G}^{1/2}_h(\tilde{G}_h + A_h A^*_h)^{-1} ( y^\delta - A_h |z^{h,\delta}_{k}|)\|_{\mathcal{Y}},
\end{array}
\end{equation}
where $\tau_0>1$ is a fixed number. It is clear that the stopping rule (\ref{Method1discrepancy}) yields a specific case of (\ref{Method1discrepancy2}) with $G=\mu I$ and $\tau_0=\tau \cdot \mu$. The numerical realization for (\ref{Method1discrepancy2}) takes the following form: ($k\leq k^*$)
\begin{equation*}\label{Method1discrepancyD}
\hspace{-1mm} \|\mathbf{G}^{\frac{1}{2}}(\mathbf{G} + \mathbf{A}_h \mathbf{A}_h^T )^{-1} ( \mathbf{y}^\delta - \mathbf{A}_h |\mathbf{z}^{k^*}|)\|_2 \leq  \frac{\delta + h C^\dagger}{\|\mathbf{G}^{\frac{1}{2}}\|} < \|\mathbf{G}^{\frac{1}{2}}(\mathbf{G} + \mathbf{A}_h \mathbf{A}_h^T )^{-1} ( \mathbf{y}^\delta - \mathbf{A}_h |\mathbf{z}^{k}|)\|_2.
\end{equation*}

To assess the accuracy of the approximate solutions, we define the $L^2$-norm relative error for an approximate solution $\hat{x}^{h,\delta}_{k^*}:=f_+\left( \sum^n_{j=1} [\mathbf{z}^{k^*}]_j \, \psi_j(k_a,k_d) \right)$: $$\textrm{L2Err}:= \|\hat{x}^{h,\delta}_{k^*} - x^\dagger\|_{L^2(\Omega)}/\|x^\dagger\|_{L^2(\Omega)},$$ where $x^\dagger$ is the the phantom, which has been used to generate data. In all simulations, we set $f_+(\cdot)=|\cdot|$ and $f_+(\cdot)=(\cdot+|\cdot|)/2$  for Algorithms 1 and 2, respectively.

All the computations were performed on a dual core personal computer with 8.00 GB RAM with MATLAB version R2019b. All experiments in this subsection are implemented for the following two examples:

\textbf{Example 1}: $\Omega:=[0,3]\times[0,3]$, $\Theta:=[0,5]\times[0.001,2]$, $t_0=0, \Delta t=0.1, t_{inj}=2$,
$x^\dagger(k_a,k_d)\equiv1$. The measurements are computed on a
mesh with 49 nodes and 72 elements.


\textbf{Example 2}: $\Omega:=[0,9]\times[0,2]$, $\Theta:=[0,8]\times[0.01,1]$, $t_0=0, \Delta t=0.2, t_{inj}=4$,
$x^\dagger(k_a,k_d)=0.5 \left\{  e^{-8[(k_a-3)^2 + (k_d-0.5)^2]} + e^{-32[(k_a-6)^2 + (k_d-1.5)^2]} \right\}$. The measurements are computed on a
mesh with 100 nodes and 162 elements.

\subsubsection{Influence of the preconditioning operator $G$.}

The purpose of this subsection is to explore the dependence of the solution accuracy
and the convergence speed on the preconditioning operator $G$ (for simplicity, we only consider its finite dimensional analogue $\mathbf{G}$), and thus to give a guide on the choices for it in practice. For focusing on the effect of these model parameters on the Algorithms 1 and 2, we fix $h'=\delta'= 0.1\%$
in this subsection. Further, we set $x_0=0$, $C^\dagger=1.1$, $\tau_0=1.1$ and $N_{max}=1,000,000$.

\begin{table}[!b]
{\footnotesize
\caption{The iterative number $k^*$ and the corresponding relative error L2Err vs $\mathbf{G}$. $h'=\delta'= 0.1\%$. $C^\dagger=1.1, \tau_0=1.1$ in Algorithms 1 and 2, and $\alpha_k=1/k$ in Algorithm 2.}\label{tab:EkVsQ}
\begin{center}
\begin{tabular}{|c|c|c|c|c|c|c|c|c|c|c|} \hline
\multirow{3}{*}{$\mathbf{G}$} &
\multicolumn{4}{c|}{\textbf{Algorithm 1}} & \multicolumn{4}{c|}{\textbf{Algorithm 2}} \\
\cline{2-9} &
\multicolumn{2}{c|}{\textbf{Example 1}} &
\multicolumn{2}{c|}{\textbf{Example 2}} &
\multicolumn{2}{c|}{\textbf{Example 1}} &
\multicolumn{2}{c|}{\textbf{Example 2}} \\
\cline{2-9}
 & L2Err & $k^*$ & L2Err & $k^*$ & L2Err & $k^*$ & L2Err & $k^*$  \\ \hline
$\mathbf{G}_1$  &  0.0138 & $N_{\max}$ & 0.0006 & 228910 &  0.0009 & $N_{\max}$ & 0.0037 & $N_{\max}$  \\
$\mathbf{G}_2$  &  0.0086 & $N_{\max}$ & 0.0013 & 64526 &  2.0745e-5 & $N_{\max}$ & 0.0002 & 129082  \\
$\mathbf{G}_3$  & 0.0003 & 188765 & 0.0467 & 122507 &  8.8714e-5 & $N_{\max}$ & 0.0243 & 594791  \\
$\mathbf{G}_4$  &   0.0002 & 24696 & 0.0506 & 13537 &  0.0004 & 37974 & 0.0293  & 41965  \\
$\mathbf{G}_5$  &  0.0318 & 20647 & 0.0022 & 35901 &  0.0229 & 27229 & 0.0012 & 75392 \\
$\mathbf{G}_6$  &  0.0649 & 38976 & 0.0562 & 7626 &  0.0142 & 52076 & 0.0116 & 38853 \\
$\mathbf{G}_7$  &  0.0002 & 79863 & 0.0074 & 13138 &  0.0003 & 56564 & 0.0016 & 67004 \\
$\mathbf{G}_8$  &  0.0570 & 12326 & 0.0526 & 18004 &  0.0215 & 24315 & 0.0159 & 91825 \\
\hline
\end{tabular}
\end{center}
}
\end{table}

Let $n$ be the dimensionality of the finite element space for the approximate solution. We study the following examples of the finite dimensional analogue $\mathbf{G}$ of $G$: $\mathbf{G}_{1}=10^{-6}\lambda_{max} I_n$, $\mathbf{G}_{2}=10^{-4} \lambda_{max} I_n$, $\mathbf{G}_{3}=10^{-3}\lambda_{max} I_n$, $\mathbf{G}_{4}=10^{-2}\lambda_{max} I_n$, $\mathbf{G}_{5}= \textrm{diag}_n(10^{-4}\lambda_{max})$, $\mathbf{G}_{6}= \textrm{diag}_n(10^{-3}\lambda_{max})$, $\mathbf{G}_{7}=U_n \mathbf{G}_{5} U^T_n$, $\mathbf{G}_{8}=U_n \mathbf{G}_{6} U^T_n$, where $\lambda_{max}=\sqrt{\|\mathbf{A}^T_h \mathbf{A}_h\|_2}$ is the maximal eigenvalue of $\mathbf{A}_h$, $\textrm{diag}_n(a)$ denotes a diagonal matrix with the minimal diagonal element $a>0$ and $n-1$ random number between $a$ and $n\cdot a$. $U_n$ denotes a random orthogonal matrix of size $n$. The detailed L2-norm relative errors `L2Err' and the corresponding iterative numbers $k^*$ for both examples are shown in Table \ref{tab:EkVsQ}, which show that the regularized approximate solution by our method not only depends on the spectral of parameter operator $\mathbf{G}$, but also depends on its detailed structure. For the trade-off between the solution accuracy and the iterative number, it is indicated in Table 2 that the best choice of $\mathbf{G}$ is in the form of $\mathbf{G}_{7}=U_n \textrm{diag}_n(10^{-4}\lambda_{max}) U^T_n$ for Example 1 and $\mathbf{G}_{2}=10^{-4} \lambda_{max} I_n$ for Example 2, which will be adapted in the following simulations.

\subsubsection{Comparison with the projected Landweber iteration.}

In this subsection, we compare the behaviors regarding the solution accuracy and the convergence rate between our methods (Algorithms 1 and 2 with the stopping rule (\ref{Method1discrepancy2})) and the projected Landweber iteration (\ref{Landweber}) with both Morozov's conventional discrepancy principle with the same $\tau_0=1.1$: (``Landweber P1'' for short)
\begin{equation*}\label{LandweberDiscrepancy}
\|A_h x^{h,\delta}_{k^*}-y^\delta\|_{\mathcal{Y}} \leq \tau_0(h+\delta)  < \|A_h x^{h,\delta}_{k}-y^\delta\|_{\mathcal{Y}}, \quad \textrm{~for all~} k< k^*,
\end{equation*}
and the newly developed discrepancy principle (\ref{Method1discrepancy2}) (``Landweber P2'' for short). The consideration of ``Landweber P2'' is mainly used for a fair comparison with our approaches under the same stopping rule. The relaxation parameter in projected Landweber methods, cf. (\ref{Landweber}), is set as $\omega=1$. The parameters in our methods are: $x_0=0$, $\mathbf{G}=\mathbf{G}_{7}$ for Example 1 and $\mathbf{G}=\mathbf{G}_{2}$ for Example 2, $C^\dagger=1.1$, $\alpha_k=1/k$ and $\tau_0=1.1$. The results of the simulations are presented in Table \ref{tab:comparison}, from which
we conclude that compared with the conventional Landweber method, Algorithms 1 and 2 provide better accuracy with considerably fewer iterations.

\begin{table}[!tbh]
{\footnotesize
\caption{Comparison with the projected Landweber methods. The CPU time is measured in seconds.}\label{tab:comparison}
\begin{center}
\begin{tabular}{|c|c|c|c|c|c|c|c|c|c|} \hline
$(h',\delta')$ &
\multicolumn{3}{c|}{$(0.1\%,0.1\%)$} &
\multicolumn{3}{c|}{$(1\%,1\%)$} &
\multicolumn{3}{c|}{$(5\%,5\%)$} \\ \hline
\multicolumn{10}{|c|}{\textbf{Example 1}} \\  \hline
\cline{2-10}
Methods   & L2Err & $k^*$ & CPU  & L2Err & $k^*$ & CPU & L2Err & $k^*$ & CPU \\ \hline
Landweber P1& 0.4310 & $N_{\max}$  & 3.6142e3  & 0.4528 & 370895  & 395.3281 & 0.5158 & 1130 & 0.0156   \\
Landweber P2& 0.4310 & $N_{\max}$ & 3.6257e3  & 0.4905 & 63599 & 2.3281  & 0.4964 & 43438 & 1.2813   \\
Algorithm 1    & 0.0002 & 79863 	 & 44.7344   & 0.0008 & 63602   & 34.7969  & 0.0053 & 43438  & 19.5625     \\
Algorithm 2    & 0.0003 & 56235  & 43.6212   & 0.0005 & 62941  & 47.3762  & 0.0021 & 60257   & 42.8194    \\ \hline
\multicolumn{10}{|c|}{\textbf{Example 2}}  \\  \hline
Methods   & L2Err & $k^*$ & CPU  & L2Err & $k^*$ & CPU & L2Err & $k^*$ & CPU \\ \hline
Landweber P1& 0.9285 & 229498  & 1.0150e3 & 0.9360 & 57647   & 44.6563  & 0.9630 & 13     & 0.1719    \\
Landweber P2& 0.9611 & 1989   & 1.0313  & 0.9615 & 1573   & 0.7656  & 0.9619 & 1055   & 0.5469   \\
Algorithm 1   & 0.0007 & 1999 & 4.4219    & 0.0030 & 1575   & 3.4063   & 0.0195 & 1059   & 2.4375   \\
Algorithm 2   & 0.0002 & 3432   & 5.0292       & 0.0016 & 2162   & 5.0594  & 0.0025 & 4284     & 5.6638  \\ \hline
\end{tabular}
\end{center}
}
\end{table}


\begin{figure}[!t]
\centering{
\includegraphics[width=1.08\textwidth]{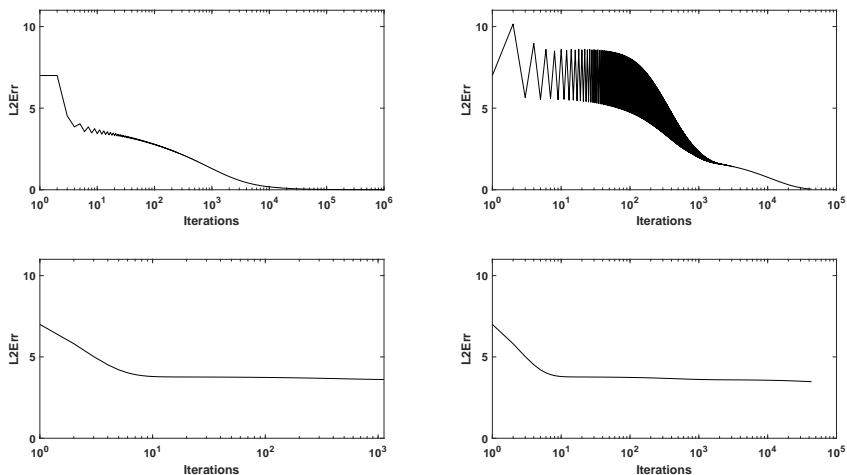}
\caption{The evolution of L2-norm relative errors `L2Err' for different methods for Example 1 with noise levels  $h'=\delta'= 5\%$. Upper (left): Algorithm 2; Upper (right): Algorithm 1; Lower (left): Landweber P1; Lower (right): Landweber P2. }
}
\label{fig:evolution}
\end{figure}

It should be noted that the use of $N_{\max}$ and the large value of obtained errors ``L2Err'' for projected Landweber methods P1 and P2 do not imply the divergence of the employed methods. Indeed, the numerical experiments in this work, see e.g. Fig. 1, indicate that all four methods are convergent. For the investigated model problems, though the approximate solutions by the projected Landweber methods converge monotonically to an exact solution, the convergence speed is extremely slow. Unlike projected Landweber methods, the two newly introduced non-negativity preserving iterative methods belong to another class of regularization methods. The resulting approximate solutions are not monotonically convergent to the exact solution; there are oscillatory phenomena in the desired approximate solutions. Fortunately this oscillation can speed up the convergence of the designed regularized solution, see the upper two figures in Fig. 1.

\subsection{Real data application}

In this section, Algorithms 1 and 2 are tested on actual experimental data~-- parathyroid hormone (PTH)~-- of the Biosensor tomography, discussed in Section 5. In the experiments the human PTH1R receptor was immobilized on a LNB-carboxyl biosensor chip using amine coupling according to manufacturer's instructions. Using the flow rate 50 $\mu L/ min$ at 20.0$^{\circ}$C we did injections of the peptide PTH(1-34) at 10 concentration levels from 5 to 220 nM. The sensorgrams were measured using a QCM biosensor Attana Cell 200 instrument. With the same preconditioning parameter $\mathbf{G}$, Algorithms 1 and 2 produce almost the same approximate RCD. Here, for the concision of the statement, we only present the results by Algorithms 1. The parameters used in this simulation are: $h'=\delta'= 1\%$, $\mathbf{G}=\mathbf{G}_2$ and $\tau_0=1.1$. Algorithm 1 requires 42,943 iterations and takes 417.0469 seconds of CPU time. The total residual error equals 1.6773. The estimated RCD and difference between the simulated individual responses and experiment data are displayed in Figures 2 and 3, respectively. By Fig. 2, our method offers more than one local peak in the reconstructed RCD, and hence one can conclude from the reconstructed RCD that there exist additional interactions in the studied chemical system. Furthermore, according to Fig. 2 one finds three local peaks in the reconstructed RCD, and hence there may exist a third reaction. This is the first time that a mathematical/computer theory supports the existence of three interactions for a PTH biological system. Though biologists have already predicted that more than two interactions could exist for a PTH biological system, due to the limitations of existing software, this prediction could not be verified in practice. It should be noted that in comparison with regression analysis (also known as the parallel reactions model), our methods could accurately resolve the three underlying interactions without \emph{a priori} assuming the existence of parallel reactions. Although more experiments are needed to fully understand the potential as well as the limitations of our approaches, their initial applications in several artificial problems and the real PTH system are quite promising. Hence, the proposed approaches in this paper offer a useful mathematical tool for the theoretical study of biosensor tomography.

\begin{figure}[!hbt]
\centering{
\includegraphics[width=0.7\textwidth]{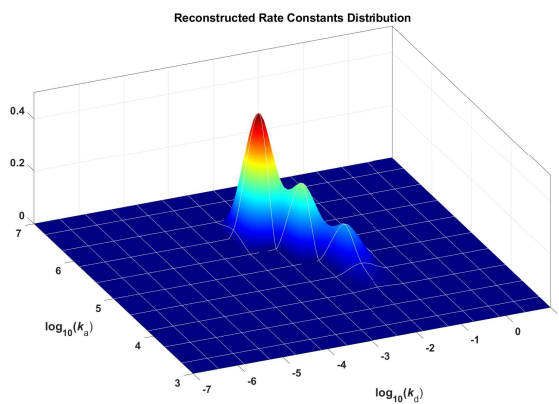}
\caption{The estimated rate constant distribution by Algorithm 1.}
}
\label{fig:evolution}
\end{figure}

\begin{figure}[!hbt]
\centering{
\includegraphics[width=0.7\textwidth]{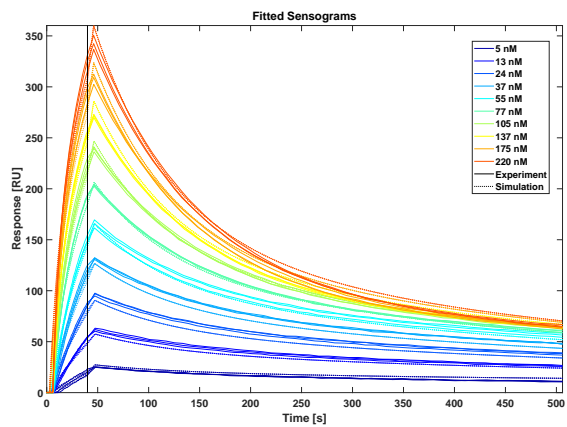}
\caption{The measured individual responses and the simulated responses by Algorithm 1. }
}
\label{fig:evolution}
\end{figure}

\section{Conclusions}

In this paper, we have presented two novel iterative regularization methods for linear ill-posed problems under non-negativity constraints. Unlike projection based methods, our new approaches preserve the non-negativity of approximate solutions $x^{h,\delta}_k$ during iterations. Strong convergence of the suggested iterative methods has been shown. Moreover, under appropriate source conditions, convergence rate results are presented for one of the approaches. The developed methods are applied for the solution of a two dimensional linear Fredholm integral equation of the first kind, which is a model of biosensor tomography. Numerical experiments of two model problems demonstrate that the presented methods are clearly faster than the projected Landweber iteration. A real data problem also indicates that our methods are able to produce a meaningful featured approximate solution that can be used in practice. Similar to the projected Landweber iteration, the two new non-negativity preserving iterative regularization methods that are introduced should also be applicable to solve non-linear ill-posed problems. This will be one of the topics of our future work.

\section*{Acknowledgments}


The authors very appreciate the fruitful discussion with X. Cheng (Zhejiang University) and are particularly grateful that he pointed out the mistake in the proof of Proposition 1 in the early version of this manuscript. We also express our thanks to T. Fornstedt and P. Forss\'{e}n (Karlstad University) for the helpful discussion on the biosensor modeling and the PTH experiments. The work of Y. Zhang is supported by the Swedish Knowledge Foundation (No. 20170059) and the Alexander von Humboldt foundation, and the work of B. Hofmann is supported by the German Research Foundation (DFG-grant HO 1454/12-1).


\bibliographystyle{AIMS}
\bibliography{ZhangHofmann2020}

\section*{Appendix: Proof of Lemma \ref{lemma_operator}.} 
 
By using the triangle inequality, we deduce that,
\begin{equation*}
\begin{array}{ll}
& \|A^*_h A_h - A^* A\| \leq \|A^*_h A_h - A^*_h A\| + \|A^*_h A - A^* A\|
\\ & \quad \leq (\|A_h \|_{L^2(\Omega)\to\mathcal{Y}} + \|A\|_{L^2(\Omega)\to\mathcal{Y}}) h \leq (2\|A \|_{L^2(\Omega)\to\mathcal{Y}} + h) h
\end{array}
\end{equation*}
which yields the first inequality (\ref{lemma_operator_ineq1}) with, $$C_1= 12 \|A \|_{L^2(\Omega)\to\mathcal{Y}} / \|G + A^*A\|$$ for $h_0< \min\left\{ \|A \|_{L^2(\Omega)\to\mathcal{Y}}, \frac{\|G + A^*A\|}{3\|A\|_{L^2(\Omega)\to\mathcal{Y}}} \right\}$  according to the following inequalities,
\begin{equation*}
\begin{array}{ll}
& \hspace{-5mm} \|(G + A^*_hA_h)^{-1} (G - A^*_hA_h) - (G + A^*A)^{-1}(G - A^*A)\|
\\ & \hspace{-5mm}  = \|(G + A^*_hA_h)^{-1} (G + A^*A)^{-1} \cdot \left[2(A^*A-A^*_hA_h)G + A^*_hA_hA^*A - A^*A A^*_hA_h \right] \|
\\ & \hspace{-5mm}  \leq \|G + A^*_hA_h\|^{-1} \|G + A^*A\|^{-1}
\\ & \hspace{-3mm}  \cdot \{ 2 \|G\| \|A^*_h A_h - A^* A\| +  \|A^*_hA_hA^*A- A^*AA^*A\| + \|A^*AA^*A-A^*A A^*_hA_h\| \}
\\ & \hspace{-5mm}  \leq 2 \|G + A^*A\|^{-1} (\|G + A^*A\|- \|A^* A - A^*_h A_h\|)^{-1}  (\|G\|+ \|A^*A\| ) \|A^*_h A_h - A^* A\|
\\ & \hspace{-5mm} \leq 2 \|G + A^*A\|^{-1} (\|G + A^*A\|-(2\|A\|_{L^2(\Omega)\to\mathcal{Y}} + h_0) h_0)^{-1}
\\ & \hspace{-3mm} \cdot (\|G\|+ \|A^*A\| ) (2\|A\|_{L^2(\Omega)\to\mathcal{Y}} + h_0) h .
\end{array}
\end{equation*}
The second inequality can be proven in the same way.

\medskip
\medskip

\end{document}